\documentclass[
final, 3p, 
]{elsarticle}
\pdfoutput=1

\usepackage[utf8]{inputenc}
\usepackage[english]{babel}

\usepackage[plainpages=false,pdfpagelabels]{hyperref}
\hypersetup{pdfauthor={Hendrik Ranocha}}
\hypersetup{pdftitle={Extended Skew-Symmetric Form for Summation-by-Parts Operators and Varying Jacobians}}
\hypersetup{pdfkeywords={hyperbolic conservation laws, high order methods,
summation-by-parts, SBP, skew-symmetric form,
correction procedure via reconstruction, CPR, entropy stability}}

\usepackage{amssymb}
\usepackage{commath}
\usepackage{siunitx}
\usepackage{mathtools}

  \usepackage{amsthm}
\theoremstyle{plain}
  \newtheorem{theorem}{Theorem}
  
\theoremstyle{definition}
  \newtheorem{remark}[theorem]{Remark}

\usepackage[small]{caption}
\usepackage{subcaption}
\usepackage{color}

\begingroup\expandafter\expandafter\expandafter\endgroup
\expandafter\ifx\csname pdfsuppresswarningpagegroup\endcsname\relax
\else
\fi

\usepackage{lineno}
\newcommand*\patchAmsMathEnvironmentForLineno[1]{%
  \expandafter\let\csname old#1\expandafter\endcsname\csname #1\endcsname
  \expandafter\let\csname oldend#1\expandafter\endcsname\csname end#1\endcsname
  \renewenvironment{#1}%
     {\linenomath\csname old#1\endcsname}%
     {\csname oldend#1\endcsname\endlinenomath}}%
\newcommand*\patchBothAmsMathEnvironmentsForLineno[1]{%
  \patchAmsMathEnvironmentForLineno{#1}%
  \patchAmsMathEnvironmentForLineno{#1*}}%
\AtBeginDocument{%
\patchBothAmsMathEnvironmentsForLineno{equation}%
\patchBothAmsMathEnvironmentsForLineno{align}%
\patchBothAmsMathEnvironmentsForLineno{flalign}%
\patchBothAmsMathEnvironmentsForLineno{alignat}%
\patchBothAmsMathEnvironmentsForLineno{gather}%
\patchBothAmsMathEnvironmentsForLineno{multline}%
}

\renewcommand{\vec}[1]{\underline{#1}}
\NewDocumentCommand{\mat}{mo}{%
  \IfValueTF{#2}{%
    \underline{\underline{#1}}{#2}
  }{%
    \underline{\underline{#1}}\,
  }%
}
\DeclarePairedDelimiter{\diagfences}{(}{)}
\newcommand{\diag}{\operatorname{diag}\diagfences}
\newcommand{\proj}{\operatorname{proj}}

\newcommand{\fnum}{f^\mathrm{num}}
\newcommand{\vecfnum}{\vec{f}^\mathrm{num}}

\begin{document}

\parindent0pt

\begin{frontmatter}

\title{Extended Skew-Symmetric Form for Summation-by-Parts Operators
       and Varying Jacobians}

\date{March 28, 2017}
\journal{Journal of Computational Physics}

\author[myu]{Hendrik Ranocha\corref{cor1}}
\ead{h.ranocha@tu-bs.de}

\author[myu]{Philipp \"Offner}
\author[myu]{Thomas Sonar}

\cortext[cor1]{Corresponding author}

\address[myu]{TU Braunschweig, Institute Computational Mathematics,
Pockelsstra\ss e 14, 38106 Braunschweig, Germany }

\begin{abstract}
  A generalised analytical notion of summation-by-parts (SBP) methods is proposed,
extending the concept of SBP operators in the correction procedure via
reconstruction (CPR), a framework of high-order methods for conservation laws.
For the first time, SBP operators with dense norms and not including boundary
points are used to get an entropy stable split-form of Burgers' equation.
Moreover, overcoming limitations of the finite difference framework, stability
for curvilinear grids and dense norms is obtained for SBP CPR methods by using
a suitable way to compute the Jacobian.

\end{abstract}


\end{frontmatter}

\section{Introduction}
\label{sec:introduction}

Conservation laws can be used to model many physical phenomena. However, the
efficient numerical solution of these equations is difficult, since stability
becomes difficult to guarantee for high-order methods.

Here, the concept of \emph{summation-by-parts} (SBP) operators in the
\emph{correction procedure via reconstruction} (CPR) framework is extended by a
new analytical notion of these schemes. For the first time, both nodal and modal
SBP bases not including boundary points can be used to construct entropy stable
and conservative approximations of Burgers' equation using an extended
skew-symmetric form.
Moreover, these dense norm bases can be coupled with curvilinear grids by a
suitable way to compute the Jacobian in SBP CPR methods, contrary to classical
finite difference SBP schemes.

SBP operators originate in the \emph{finite difference} (FD) framework and yield
an approach to prove stability in a way similar to the continuous investigations
by mimicking \emph{integration-by-parts} on a discrete level, see inter alia the
review articles \cite{svard2014review, fernandez2014review} and references cited
therein. To get conservation and stability, (exterior and inter-element)
boundary conditions are imposed weakly by \emph{simultaneous approximation terms}
(SATs) and skew-symmetric formulations for nonlinear conservation laws are
used \cite{fisher2013discretely}.

Gassner \cite{gassner2013skew} applied the SBP framework to a
\emph{discontinuous Galerkin} (DG) spectral element method (DGSEM) using the
nodes of Lobatto-Legendre quadrature. Additionally,
Fernández et al. \cite{fernandez2014generalized} proposed an extended definition of SBP methods
in a numerical framework relying on nodal representations.

Ranocha et al. \cite{ranocha2016summation} investigated connections between SBP
methods and the general CPR framework, unifying flux reconstruction \cite{huynh2007flux}
and lifting collocation penalty \cite{wang2009unifying} schemes.
Several high-order methods such as DG, spectral volume and spectral difference
methods can be formulated in this framework, as described in the review article
\cite{huynh2014high} and references cited therein.

In this article, a brief review of SBP CPR methods is given in section \ref{sec:SBP},
followed by an introduction to some results about a skew-symmetric splitting 
for diagonal norm bases.

The main contribution of this article will be presented in section \ref{sec:generalisation}.
There, a more general setting for SBP CPR methods will be described.
Based on this, an extended skew-symmetric form of Burgers' equation is proposed.
Extending the correction terms in this form, conservation and nonlinear entropy
stability are proved for general SBP CPR semidiscretisations, including both nodal
bases without boundary nodes (e.g. Gauß-Legendre nodes) and modal bases
(e.g. Legendre polynomials). Numerical examples are presented thereafter.
Additionally, a brief comparison with the numerical setting of Fernández et al.
\cite{fernandez2014generalized} is given.

Moreover, limitations of the application of dense norms for curvilinear grids known
in FD SBP methods are overcome for SBP CPR methods in section \ref{sec:curvilinear}.

Finally, the results are summarised in section \ref{sec:summary}, followed by a
discussion and further topics of research.

\section{Correction procedure via reconstruction using diagonal-norm SBP operators}
\label{sec:SBP}

In this section, the basic concept and results about summation-by-parts operators
for correction procedure via reconstruction of \cite{ranocha2016summation} are
briefly reviewed.

CPR schemes are designed as semidiscretisations of hyperbolic conservation laws
\begin{equation}
\label{eq:conservation-law}
  \partial_t u + \partial_x f(u) = 0,
\end{equation}
equipped with appropriate initial and boundary conditions.
The domain is divided into non-overlapping intervals and each interval is mapped
to the reference element $[-1,1]$ for the computations. In each element, a nodal
polynomial basis  $\mathcal{B}$ of order $p$ is used to represent the numerical
solution.
The semidiscretisation of \eqref{eq:conservation-law} (i.e.
the computation of $\partial_x f(u)$) consists of the following steps, see also
the review \cite{huynh2014high} and references cited therein:
\begin{itemize}
  \item
    Interpolate the solution $u$ to the cell boundaries at $-1$ and $1$ (if these
    values are not already given as coefficients of the nodal basis).
  
  \item
    Compute common numerical fluxes $\fnum$ at each cell boundary.
  
  \item
    Compute the flux $f(u)$ pointwise in each node.
    
  \item
    Interpolate the flux $f(u)$ to the boundary and add polynomial correction
    functions $g_L$, $g_R$ of degree $p+1$, multiplied by the difference
    $f_{L/R} - \fnum_{L/R}$ of the flux and the numerical flux at the
    corresponding boundary.
    
  \item
    Finally, compute the resulting derivative of
    $f + (f_{L} - \fnum_{L}) g_L + (f_{R} - \fnum_{R}) g_R$,
    using exact differentiation for the polynomial basis.
\end{itemize}

For the representation of a diagonal-norm SBP operator, the basis $\mathcal{B}$ has to be
associated with a quadrature rule, given by nodes $z_0, \dots, z_p$ and
appropriate positive weights $\omega_0, \dots, \omega_p$. The values of $u$ at
the nodes are the coefficients of the local expansion, i.e.
$\vec{u} = (u(z_0), \dots, u(z_p))^T$. The quadrature weights determine a
positive definite mass matrix $\mat{M} = \diag{\omega_0, \dots, \omega_p}$ associated
with a discrete norm $\norm{u}_M^2 = \vec{u}^T \mat{M} \vec{u}$.
Moreover, the derivative is represented by the matrix $\mat{D}$ and the
restriction to the boundary $\set{-1,1}$ of the standard element $[-1,1]$ is
performed by the restriction matrix $\mat{R}$. For nodal bases including boundary
points (e.g. Lobatto-Legendre nodes), it is given by $\mat{R} = \begin{pmatrix}
1 & 0 & \dots & 0 \\ 0 & \dots & 0 & 1 \end{pmatrix}$. The bilinear form giving the
difference of boundary values is represented by the matrix $\mat{B} = \diag{-1,1}$.

The basis and its associated quadrature rule must satisfy the SBP property
\begin{equation}
\label{eq:SBP}
  \mat{M} \mat{D} + \mat{D}[^T] \mat{M} = \mat{R}[^T] \mat{B} \mat{R},
\end{equation}
in order to mimic integration by parts on a discrete level
\begin{equation}
\begin{aligned}
  \int_{-1}^{1} u \, (\partial_x v) \dif x + \int_{-1}^{1} (\partial_x u) \, v \dif x
  \approx
  \vec{u}^T \mat{M} \mat{D} \vec{v} + (\mat{D} \vec{u})^T \mat{M} \vec{v}
  =
  (\mat{R} \vec{u})^T \mat{B} (\mat{R} \vec{v})
  \approx
  \eval[2]{u \, v}_{-1}^{1}.
\end{aligned}
\end{equation}
If the quadrature is exact for polynomials of degree $\leq 2p-1$, this condition
is fulfilled, since all integrals are evaluated exactly, see also
\cite{kopriva2010quadrature, hicken2013summation}.

\cite{ranocha2016summation} introduced a formulation of CPR methods with special
attention paid to SBP operators.
After mapping each element to the standard element $[-1,1]$,
a CPR method can be formulated as
\begin{equation}
  \partial_t \vec{u} + \mat{D} \vec{f} + \mat{C} ( \vecfnum - \mat{R} \vec{f}) = 0.
\label{eq:CPR}
\end{equation}
Thus, for a given standard element, a CPR method is parametrised by
\begin{itemize}
  \item 
  A basis $\mathcal{B}$ for the local expansion, determining the derivative and
  restriction (interpolation) matrices $\mat{D}$ and $\mat{R}$.
  
  \item
  A correction matrix $\mat{C}$, adapted to the chosen basis.
\end{itemize}

As an example, consider Gauß-Lobatto-Legendre integration with its associated
basis of point values at Lobatto nodes in $[-1, 1]$. 
Using the special choice $\mat{C} = \mat{M}[^{-1}] \mat{R}[^T] \mat{B}$ and defining
$\mat{\widetilde{B}} := \mat{R}[^T] \mat{B} \mat{R}$, i.e.
$\mat{\widetilde{B}} = \diag{-1, 0, \dots, 0, 1}$,
the CPR method of equation \eqref{eq:CPR} reduces to
\begin{equation}
  \partial_t \vec{u} + \mat{D} \vec{f}
  + \mat{M}[^{-1}] \mat{\widetilde{B}}
    \left( \vec{\widetilde{f}}^\mathrm{num} - \vec{f} \right)
  = 0,
\label{eq:DGSEM}
\end{equation}
where $\vec{\widetilde{f}}^\mathrm{num} = (\fnum_L, 0, \dots, 0, \fnum_R)$ contains
the numerical flux at the left and right boundary and satisfies
$\vecfnum = \mat{R} \vec{\widetilde{f}}^\mathrm{num}$. Equation \eqref{eq:DGSEM}
is the strong form of the DGSEM formulation of Gassner \cite{gassner2013skew},
which he proved to be a diagonal norm SBP operator.

\subsection{Results about the skew-symmetric form of Burgers' equation and
diagonal-norm SBP operators}

Stability properties for linear and nonlinear problems can be very different. 
\cite{ranocha2016summation} considered Burgers' equation
\begin{equation}
  \partial_t u + \partial_x \frac{u^2}{2} = 0
\label{eq:inviscid-burgers}
\end{equation}
in one space dimension with periodic boundary conditions and appropriate initial
condition.

As described by \cite{gassner2013skew}, a split-operator form of the flux divergence
$\alpha \partial_x \frac{1}{2} u^2 + (1-\alpha) u \partial_x u$
can be used to get conservation and stability (across elements) if boundary nodes
are included in the basis and the numerical flux is entropy stable in the sense
of Tadmor \cite{tadmor1987numerical,tadmor2003entropy}, i.e.
$\frac{1}{6} ( u_-^3 - u_+^3 ) - (u_- - u_+) \fnum(u_-, u_+) \leq 0$.
Here, the numerical flux $\fnum(u_-, u_+)$ is computed given the values
of $u$ from the elements to the left $(u_-)$ and to the right ($u_+$) of a given
boundary. This strong form discontinuous Galerkin spectral element method of 
\cite{gassner2013skew} with the choice $\alpha = \frac{2}{3}$ can be written as
the following semidiscretisation of the skew-symmetric form of Burgers' equation
\begin{equation}
  \partial_t \vec{u}
  + \mat{D} \frac{1}{2} \vec{u^2}
  + \underbrace{
    \frac{1}{3} \left( \mat{u} \mat{D} \vec{u} - \frac{1}{2} \mat{D} \mat{u} \vec{u} \right)
    }_{=:\vec{c}_\mathrm{div}}
  + \mat{M}[^{-1}] \mat{R}[^T] \mat{B} \left(
    \vecfnum - \mat{R} \frac{1}{2} \vec{u^2}
  \right)
  = \vec{0},
\end{equation}
where the vector $\vecfnum = (\fnum_L, \fnum_R)$ contains the numerical fluxes at
the left and right boundaries of the element. The term $\vec{c}_\mathrm{div}$ is a a discrete
correction to the divergence term, that has to be used since the product rule is
not valid discretely.

For a general SBP basis without boundary nodes, the stability investigation
is more complicated. \cite{ranocha2016summation} introduced and analysed a new
correction term $\vec{c}_\mathrm{res}$ for the restriction to the boundary, resulting in
\begin{gather}
  \label{eq:inviscid-burgers-CPR-full-correction}
  \partial_t \vec{u}
  + \mat{D} \frac{1}{2} \vec{u^2}
  + \vec{c}_\mathrm{div}
  + \mat{M}[^{-1}] \mat{R}[^T] \mat{B} \left(
    \vecfnum - \mat{R} \frac{1}{2} \vec{u^2} - \vec{c}_\mathrm{res}
  \right)
  = \vec{0},
  \\
  \label{eq:inviscid-burgers-CPR-correction-terms}
  \vec{c}_\mathrm{div}
  =
  \frac{1}{3} \left(
    \mat{u} \mat{D} \vec{u} - \frac{1}{2} \mat{D} \mat{u} \vec{u}
  \right)
  ,\qquad
  \vec{c}_\mathrm{res}
  =
  \frac{1}{6} \left(
    (\mat{R} \vec{u})^2 - \mat{R} \mat{u} \vec{u}
  \right).
\end{gather}
It has been stressed that not only the product rule is not valid discretely, but
multiplication is not exact. This results in incorrect divergence and restriction
terms that have to be corrected in order to get the desired conservation and
stability estimates.
\begin{theorem}[Theorem 9 of \cite{ranocha2016summation}]
\label{thm:CPR-burgers}
  If the numerical flux $\fnum$ satisfies
  \begin{equation}
  \label{eq:fnum-ES}
    \frac{1}{6} ( u_-^3 - u_+^3 ) - (u_- - u_+) \fnum(u_-, u_+) \leq 0,
  \end{equation}
  then a diagonal-norm SBP CPR method 
  \eqref{eq:inviscid-burgers-CPR-full-correction} with correction terms for both
  divergence and restriction to the boundary
  \eqref{eq:inviscid-burgers-CPR-correction-terms} for the inviscid Burgers'
  equation \eqref{eq:inviscid-burgers} is both conservative and stable in the
  discrete norm $\norm{\cdot}_{M}$ induced by $\mat{M}$.
  Numerical fluxes fulfilling condition \eqref{eq:fnum-ES} are inter alia
  \begin{itemize}
    \item the energy conservative (ECON) flux
      \begin{equation}
      \label{eq:burgers-econ-flux}
        \fnum(u_-, u_+)
        =
        \frac{1}{4} ( u_+^2 + u_-^2 )
        - \frac{(u_+ - u_-)^2}{12}
      \end{equation}
    \item the local Lax-Friedrichs (LLF) flux
      \begin{equation}
      \label{eq:burgers-llf-flux}
        \fnum(u_-, u_+)
        =
        \frac{1}{4} ( u_+^2 + u_-^2 )
        - \frac{\max( |u_+|, |u_-| )}{2} (u_+ - u_-)
      \end{equation}
    \item and Osher's flux
      \begin{equation}
      \label{eq:burgers-osher-flux}
        \fnum(u_-, u_+)
        =
        \begin{dcases}
          \frac{u_-^2}{2}, & u_+, u_- > 0, \\
          \frac{u_+^2}{2}, & u_+, u_- < 0, \\
          \frac{u_+^2}{2} + \frac{u_-^2}{2}, & u_- \geq 0 \geq u_+, \\
          0, & u_- \leq 0 \leq u_+.
        \end{dcases}
      \end{equation}
  \end{itemize}
\end{theorem}

\section{Abstract view and generalisation}
\label{sec:generalisation}

The basic setting described in the previous section
uses diagonal norm SBP operators and nodal bases, associated with quadrature rules
with positive weights. These operators have been used in the context
of CPR methods to obtain conservative and stable semidiscretisations for
linear advection and Burgers' equation.
This chapter provides a more abstract view on the results and generalised
schemes with a new form of the correction terms, allowing both modal and nodal
bases with arbitrary (dense) norm.

\subsection{Analytical setting in one dimension}
\label{sec:analytical-setting-in-one-dimension}

Continuing the investigations, an analytical setting
in the one-dimensional standard element $\Omega = [-1, 1]$ is presented at first.
The semidiscretisation in space consists of the representation of a numerical
solution in a (real) finite dimensional Hilbert space $X_V$, 
the space of functions on the (one-dimensional) \emph{volume} $\Omega$.
Hitherto, $X_V$
has been the space of polynomials of degree $\leq p$, i.e. $\dim X_V = p+1$.
$X_V$ is equipped with a suitable basis $\mathcal{B}_V$, e.g. a Lagrange
(interpolation) basis for Gauß-Legendre or Lobatto-Legendre quadrature nodes.
With regard to $\mathcal{B}_V$, the \emph{scalar product and associated norm}
on $X_V$ are given by a symmetric and positive-definite matrix $\mat{M}$,
approximating the $L^2$ norm on $X_V$, i.e.
\begin{equation}
  \vec{u}^T \mat{M} \vec{v}
  = \left\langle \vec{u}, \vec{v} \right\rangle_M
  \approx
  \int_{\Omega} u v
  = \left\langle u, v \right\rangle_{L^2}.
\end{equation}
In one dimension, a \emph{divergence} (derivative) operator mapping $X_V$ to
$X_V$ is represented by a matrix $\mat{D}$.

Besides $X_V$, the vector space $X_B$ of functions on the ($0$-dimensional)
\emph{boundary} $\partial \Omega$ of the standard element $\Omega$ is used.
The associated basis is denoted by $\mathcal{B}_B$.
In the simple one-dimensional case, $X_B$ is a two-dimensional vector space
and $\mathcal{B}_B$ is chosen to represent point values at $-1$ and $1$.
On the boundary, a bilinear form is represented by a matrix $\mat{B}$,
approximating the \emph{boundary (surface) integral} in the outward normal
direction, i.e. evaluation at the boundary.
More precisely, $\mat{B}$ maps $X_B \times X_B$ to $\mathbb{R}$ and
\begin{equation}
  \vec{u}_B^T \mat{B} \vec{f}_B
  = B(u_B, f_B)
  \approx \eval[2]{u_B \, f_B}_{-1}^{1}.
\end{equation}
In the simple one-dimensional setting, $u_B$ and $f_B$ are both scalar functions
and $\int_{\partial \Omega} u_B \, f_B \cdot n = u(1) f(1) - u(-1) f(-1)$, i.e.
$\mat{B} = \diag{-1,1}$ if $\mathcal{B}_B$ is ordered such that the value at
$-1$ is the first coefficient.
With regard to the chosen bases $\mathcal{B}_V$ and $\mathcal{B}_B$, 
a \emph{restriction} operator is represented by a matrix
$\mat{R}$, mapping a function $u$ on the volume to its values at the
boundary.
Again, the SBP property
$\mat{M} \mat{D} + \mat{D}[^T] \mat{M} = \mat{R}[^T] \mat{B} \mat{R}$ \eqref{eq:SBP}
mimics integration by parts.

A CPR method is further parametrised by a \emph{correction} or \emph{penalty}
operator, represented by a matrix $\mat{C}$ adapted to the chosen bases. The
canonical choice is $\mat{C} = \mat{M}[^{-1}] \mat{R}[^T] \mat{B}$ as described in
\cite{ranocha2016summation}, especially for nonlinear equations. For linear advection,
other choices of $\mat{C}$ are possible, recovering the full range of
linearly stable schemes presented in \cite{vincent2015extended}, see
\cite[section 3]{ranocha2016summation}.

Since nonlinear fluxes $f(u)$ appear and are of interest, nonlinear operations
on $X_V$ have to be described. In general, if $X_V$ is a finite-dimensional 
vector space of polynomials containing polynomials of degree $\leq p$
($p \geq 1$ and $p$ is minimal), then the product of $u, v \in X_V$ is a
polynomial of degree $\leq 2p$, i.e. not in $X_V$ in general. Therefore, discrete
multiplication is not exact. Thus, multiplying $v \in X_V$ with $u \in X_V$
yields $\mat{u}[^+] \vec{v} \in X_V^+$, where $X_V^+ \supset X_V$ is a vector
space of higher dimension. After this exact multiplication, a 
\emph{projection} on $X_V$ is performed, resulting in $\mat{u} \vec{v} \in X_V$.

For a nodal basis $\mathcal{B}_V$, the natural projection is given by pointwise
evaluation at the nodes. However, for a modal
basis of Legendre polynomials, the natural projection is an $L^2$ orthogonal
projection on $X_V$. Disappointingly, this concept does not easily extend to division,
since $L^2$ projection of rational functions is not a simple task.

\subsection{Revisiting Burgers' equation}

Investigating again a skew-symmetric SBP CPR method without the assumption of
a nodal and/or orthogonal basis, some further complications arise. In contrast
to the manipulations used to prove Theorem \ref{thm:CPR-burgers}
(see also \cite{gassner2013skew}),
$\mat{u}$ and $\mat{M}$ might not commute, either because the nodal basis is
not orthogonal or because a modal basis is chosen.
Therefore, the correction terms
\eqref{eq:inviscid-burgers-CPR-correction-terms} for the divergence and restriction
do not suffice to prove conservation and stability. The reason is again
inexactness of discrete multiplication. A multiplication operator $\mat{u}$
should be self-adjoint, at least in a finite-dimensional space (and in general,
if a correct domain is chosen). Thus, instead of $\mat{u}$ in the first term
of $\vec{c}_\mathrm{div}$, the adjoint $\mat{u}[^*]$ of $\mat{u}$ with respect to the
scalar product induced by $\mat{M}$ is proposed. The symmetry condition 
\begin{equation}
  \left\langle \vec{v}, \mat{u} \vec{w} \right\rangle_M
  =
  \left\langle \mat{u}[^*] \vec{v}, \vec{w} \right\rangle_M
\end{equation}
can be written as
\begin{equation}
  \vec{v}^T \mat{M} \mat{u} \vec{w}
  =
  \vec{v}^T (\mat{u}[^*])^T \mat{M} \vec{w}.
\end{equation}
Thus, since $\vec{v}$ and $\vec{w}$ are arbitrary,
$ \mat{M} \mat{u} = (\mat{u}[^*])^T \mat{M}$, i.e.
$ \mat{u}[^*] = \mat{M}[^{-1}] \mat{u}[^T] \mat{M}$, and the generalised correction
terms are
\begin{equation}
  \vec{c}_\mathrm{div}
  = \frac{1}{3} \left( \mat{M}[^{-1}] \mat{u}[^T] \mat{M} \mat{D} \vec{u}
                      - \frac{1}{2} \mat{D} \mat{u} \vec{u}
                \right)
  ,\qquad
  \vec{c}_\mathrm{res}
  = \frac{1}{6} \left( (\mat{R} \vec{u})^2 - \mat{R} \mat{u} \vec{u} \right).
\label{eq:inviscid-burgers-CPR-correction-terms-general}
\end{equation}
Using these correction terms, Theorem \ref{thm:CPR-burgers} is generalised by
\begin{theorem}
\label{thm:CPR-burgers-general}
  If the numerical flux $\fnum$ satisfies
  $\frac{1}{6} ( u_-^3 - u_+^3 ) - (u_- - u_+) \fnum(u_-, u_+) \leq 0$
  \eqref{eq:fnum-ES},
  then a general SBP CPR method with $\mat{C} = \mat{M}[^{-1}] \mat{R}[^T] \mat{B}$
  and correction terms \eqref{eq:inviscid-burgers-CPR-correction-terms-general}
  for both divergence and restriction to the boundary
  \begin{equation}
    \partial_t \vec{u}
    + \mat{D} \frac{1}{2} \vec{u^2}
    + \vec{c}_\mathrm{div}
    + \mat{C} \left( \vecfnum
                    - \mat{R} \frac{1}{2} \vec{u^2}
                    - \vec{c}_\mathrm{res}
              \right)
    = \vec{0},
  \label{eq:inviscid-burgers-CPR-full-correction-general}
  \end{equation}
  for the inviscid Burgers' equation \eqref{eq:inviscid-burgers} 
  is both conservative and stable across elements
  in the discrete norm $\norm{\cdot}_{M}$ induced by $\mat{M}$.
  Numerical fluxes fulfilling condition \eqref{eq:fnum-ES} are inter alia
  \begin{itemize}
    \item the energy conservative (ECON) flux \eqref{eq:burgers-econ-flux},
    \item the local Lax-Friedrichs (LLF) flux \eqref{eq:burgers-llf-flux},
    \item and Osher's flux \eqref{eq:burgers-osher-flux}.
  \end{itemize}
\end{theorem}
\begin{proof}
  Multiplying $\partial_t \vec{u}$ with $\vec{v}^T \mat{M}$, inserting
  $\mat{C} = \mat{M}[^{-1}] \mat{R}[^T] \mat{B}$ and applying the SBP property
  \eqref{eq:SBP} yields
  \begin{equation}
  \begin{aligned}
    \vec{v}^T \mat{M} \partial_t \vec{u}
    &=
    - \frac{1}{2} \vec{v}^T \mat{M} \mat{D} \mat{u} \vec{u}
    - \vec{v}^T \mat{M} \vec{c}_\mathrm{div}
    - \vec{v}^T \mat{R}[^T] \mat{B} \left(\vecfnum
                                        - \frac{1}{2} \mat{R} \mat{u} \vec{u}
                                        - \vec{c}_\mathrm{res}
                                        \right)
    \\&=
    + \frac{1}{2} \vec{v}^T \mat{D}[^T] \mat{M} \mat{u} \vec{u}
    - \frac{1}{2} \vec{v}^T \mat{R}[^T] \mat{B} \mat{R} \mat{u} \vec{u}
    \\&\quad- \vec{v}^T \mat{M} \vec{c}_\mathrm{div}
    - \vec{v}^T \mat{R}[^T] \mat{B} \left(\vecfnum
                                        - \frac{1}{2} \mat{R} \mat{u} \vec{u}
                                        - \vec{c}_\mathrm{res}
                                        \right).
  \end{aligned}
  \end{equation}
  Gathering terms and inserting $\vec{c}_\mathrm{div}$, $\vec{c}_\mathrm{res}$ from equation
  \eqref{eq:inviscid-burgers-CPR-correction-terms-general} results in
  \begin{equation}
  \begin{aligned}
    \vec{v}^T \mat{M} \partial_t \vec{u}
    &=
      \frac{1}{2} \vec{v}^T \mat{D}[^T] \mat{M} \mat{u} \vec{u}
    - \vec{v}^T \mat{M} \vec{c}_\mathrm{div}
    - \vec{v}^T \mat{R}[^T] \mat{B} \vecfnum
    + \vec{v}^T \mat{R}[^T] \mat{B} \vec{c}_\mathrm{res}
    \\&=
      \frac{1}{2} \vec{v}^T \mat{D}[^T] \mat{M} \mat{u} \vec{u}
    - \frac{1}{3} \vec{v}^T \mat{u}[^T] \mat{M} \mat{D} \vec{u}
    + \frac{1}{6} \vec{v}^T \mat{M} \mat{D} \mat{u} \vec{u}
    \\&\quad
    - \vec{v}^T \mat{R}[^T] \mat{B} \vecfnum
    + \frac{1}{6} \vec{v}^T \mat{R}[^T] \mat{B} ( \mat{R} \vec{u} )^2
    - \frac{1}{6} \vec{v}^T \mat{R}[^T] \mat{B} \mat{R} \mat{u} \vec{u}.
  \end{aligned}
  \end{equation}
  Applying the SBP property \eqref{eq:SBP} for the third term yields
  \begin{equation}
  \begin{aligned}
    \vec{v}^T \mat{M} \partial_t \vec{u}
    &=
      \frac{1}{2} \vec{v}^T \mat{D}[^T] \mat{M} \mat{u} \vec{u}
    - \frac{1}{3} \vec{v}^T \mat{u}[^T] \mat{M} \mat{D} \vec{u}
    + \frac{1}{6} \vec{v}^T \mat{R}[^T] \mat{B} \mat{R} \mat{u} \vec{u}
    - \frac{1}{6} \vec{v}^T \mat{D}[^T] \mat{M} \mat{u} \vec{u}
    \\&\quad
    - \vec{v}^T \mat{R}[^T] \mat{B} \vecfnum
    + \frac{1}{6} \vec{v}^T \mat{R}[^T] \mat{B} ( \mat{R} \vec{u} )^2
    - \frac{1}{6} \vec{v}^T \mat{R}[^T] \mat{B} \mat{R} \mat{u} \vec{u}
    \\&=
      \frac{1}{3} \vec{v}^T \mat{D}[^T] \mat{M} \mat{u} \vec{u}
    - \frac{1}{3} \vec{v}^T \mat{u}[^T] \mat{M} \mat{D} \vec{u}
    - \vec{v}^T \mat{R}[^T] \mat{B} \vecfnum
    + \frac{1}{6} \vec{v}^T \mat{R}[^T] \mat{B} ( \mat{R} \vec{u} )^2.
  \end{aligned}
  \label{eq:CPR-burgers-v-M}
  \end{equation}
  
  In order to obtain \emph{stability},
  $\frac{1}{2} \od{}{t} \norm{u}_M^2 = \vec{u}^T \mat{M} \partial_t \vec{u}$
  has to be considered. Thus, setting $\vec{v} = \vec{u}$ in
  \eqref{eq:CPR-burgers-v-M} and using the symmetry of the $\mat{M}$ results in
  \begin{equation}
    \frac{1}{2} \od{}{t} \norm{u}_M^2
    =
    - \vec{u}^T \mat{R}[^T] \mat{B} \vecfnum
    + \frac{1}{6} \vec{u}^T \mat{R}[^T] \mat{B} ( \mat{R} \vec{u} )^2,
  \end{equation}
  Denoting the values from the cells left and right to a given boundary node
  with indices $-$ and $+$, respectively, and summing over all elements, the
  contribution of one boundary node to $\frac{1}{2} \od{}{t} \norm{u}_M^2$ is
  \begin{equation}
    \frac{1}{6} (u_-^3 - u_+^3) - (u_- - u_+) \fnum(u_-, u_+).
  \end{equation}
  If this is non-positive (as assumed), stability in the discrete norm described
  by $\mat{M}$ is guaranteed.
  
  Investigation \emph{conservation} by setting $\vec{v} = \vec{1}$ in
  \eqref{eq:CPR-burgers-v-M}, using
  $\mat{D} \vec{1} = 0$ (i.e. exact differentiation for constant functions)
  and $\mat{u} \vec{1} = \vec{u}$ (i.e. exact multiplication with constant
  functions) yields
  \begin{equation}
    \od{}{t} \vec{1}^T \mat{M} \vec{u}
    =
    - \frac{1}{3} \vec{u}^T \mat{M} \mat{D} \vec{u}
    - \vec{1}^T \mat{R}[^T] \mat{B} \vecfnum
    + \frac{1}{6} \vec{1}^T \mat{R}[^T] \mat{B} ( \mat{R} \vec{u} )^2.
  \end{equation}
  Rewriting the first term (by the SBP property \eqref{eq:SBP}) as 
  \begin{equation}
    - \frac{1}{3} \vec{u}^T \mat{M} \mat{D} \vec{u}
    =
    - \frac{1}{6} \vec{u}^T \mat{M} \mat{D} \vec{u}
    + \frac{1}{6} \vec{u}^T \mat{D}[^T] \mat{M} \vec{u}
    - \frac{1}{6} \vec{u}^T \mat{R}[^T] \mat{B} \mat{R} \vec{u}
    =
    - \frac{1}{6} \vec{u}^T \mat{R}[^T] \mat{B} \mat{R} \vec{u}
  \end{equation}
  results in
  \begin{equation}
    \od{}{t} \vec{1}^T \mat{M} \vec{u}
    =
    - \frac{1}{6} \vec{u}^T \mat{R}[^T] \mat{B} \mat{R} \vec{u}
    - \vec{1}^T \mat{R}[^T] \mat{B} \vecfnum
    + \frac{1}{6} \vec{1}^T \mat{R}[^T] \mat{B} ( \mat{R} \vec{u} )^2.
  \end{equation}
  Denoting the values of $u$ at the left and right boundary as $u_L$ and $u_R$,
  respectively,
  \begin{equation}
    \vec{u}^T \mat{R}[^T] \mat{B} \mat{R} \vec{u}
    =
    u_R \cdot u_R - u_L \cdot u_L
    = 
    1 \cdot u_R^2 - 1 \cdot u_L^2
    =
    \vec{1}^T \mat{R}[^T] \mat{B} (\mat{R} \vec{u})^2.
  \end{equation}
  and therefore
  \begin{equation}
    \od{}{t} \vec{1}^T \mat{M} \vec{u}
    =
    - \vec{1}^T \mat{R}[^T] \mat{B} \vecfnum.
  \end{equation}
  Thus, summing over all elements, the contribution of both cells sharing a
  common boundary node cancel each other, since the numerical flux $\fnum$ is
  the same for both elements.
\end{proof}

\begin{remark}
  In Theorem \ref{thm:CPR-burgers-general}, conservation across elements has
  been considered. On a sub-element level, conservation for diagonal-norm SBP
  operators including boundary nodes has been proven in \cite{fisher2013discretely}
  in the context of the Lax-Wendroff theorem. An extension to dense norm and
  modal bases is still an open question.
\end{remark}

\begin{remark}
\label{rem:bases}
  Regarding conservation and stability across elements, Theorem~\ref{thm:CPR-burgers-general}
  is very general. However, there are several special cases that deserve to be
  mentioned explicitly for comparison and a better understanding.
  \begin{enumerate}
    \item
    \emph{Nodal bases including boundary nodes with diagonal mass matrix.}
    \\
    This has been considered inter alia in \cite{gassner2013skew}.
    Since boundary nodes are included, the correction term for the restriction
    vanishes, $\vec{c}_\mathrm{res} = \vec{0}$.
    Moreover, since the mass matrix $\mat{M}$ and the multiplication operators
    $\mat{u}$ are diagonal, the $\mat{M}$-adjoint is simply
    $\mat{u}[^*] = \mat{M}[^{-1}] \mat{u}[^T] \mat{M} = \mat{u}^T = \mat{u}$.
    
    \item
    \emph{Nodal bases (possibly not including boundary nodes) with diagonal mass matrix.}
    \\
    This has been considered in \cite{ranocha2016summation}. Here, the correction
    term for the restriction is in general not zero and has to be used to get a
    conservative and stable scheme. However, the mass matrix $\mat{M}$ and the
    multiplication operators $\mat{u}$ are diagonal, resulting in $\mat{u}[^*] = \mat{u}$.
    
    \item
    \emph{Nodal bases (possibly not including boundary nodes) with general mass matrix.}
    \\
    Here, both the usage of both correction terms with the correct $\mat{M}$-adjoint
    $\mat{u}[^*] = \mat{M}[^{-1}] \mat{u}[^T] \mat{M}$ is necessary in general
    to get the desired properties of the scheme.
    
    \item
    \emph{Modal Legendre bases.}
    \\
    In this case, no boundary nodes can be ``included''. Thus, the correction term
    for the restriction to the boundary does not vanish in general. 
    Here, multiplication operators performing exact multiplication followed by
    an orthogonal projection are considered. For this orthogonal basis, a
    multiplication operator $\mat{u}$ is in general not diagonal, but
    $\mat{M}$-self-adjoint, as the following calculation for arbitrary polynomials
    $u, v, w$ of degree $\leq p$ shows:
    \begin{equation}
      \left\langle \vec{v}, \mat{u} \vec{w} \right\rangle_M
      =
      \vec{v}^T \mat{M} \mat{u} \vec{w}
      =
      \int v \proj(u \, w)
      =
      \int v \, u \, w
      =
      \int \proj(u \, v) \, w
      =
      \vec{v}^T \mat{u}[^T] \mat{M} \vec{w}
      =
      \left\langle \mat{u} \vec{v}, \vec{w} \right\rangle_M.
    \end{equation}
    In the second step, the definition of the multiplication operator as exact
    multiplication followed by an orthogonal projection on the space of polynomials
    of degree $\leq p$ is inserted. In the following step, this orthogonality is used,
    since $v$ is a polynomial of degree $\leq p$. Similarly, the orthogonality of Legendre
    polynomials is used in the fourth step.
    Thus, multiplication operators $\mat{u}$ are $\mat{M}$-self-adjoint.
  \end{enumerate}
\end{remark}

\begin{remark}
  The correction terms used in Theorem \ref{thm:CPR-burgers-general} can be extended
  to systems as well. In \cite{gassner2016well}, an entropy stable split form of
  the shallow water equations has been analysed using Lobatto-Legendre nodes.
  This has been extended to two-dimensional curvilinear grids in
  \cite{wintermeyer2015entropy} and to a whole two-parameter family of splittings
  for general SBP bases in \cite{ranocha2017shallow}. Moreover, a kinetic energy
  preserving DG method for the Euler equations using a split form has been
  proposed in \cite{ortleb2016kineticTalk, ortleb2016kinetic}.
\end{remark}

\subsection{Numerical results for dense norm and modal bases}
\label{sec:numerical-tests-Burgers}

Here, Burgers' equation \eqref{eq:inviscid-burgers} is considered in the domain
$[0, 2]$ with periodic boundary conditions . The initial condition 
\begin{equation}
  u(0, x) = u_0(x) = \sin(\pi x) + 0.01
\end{equation}
is evolved in time using the classical fourth order Runge-Kutta method with
$\num{10000}$ time steps in the time interval $[0, 3]$. As semidiscretisation in
space, several SBP CPR methods with $N = 20$ equally spaced elements describing
polynomials of degree $\leq p = 7$ and correction terms
\eqref{eq:inviscid-burgers-CPR-correction-terms-general} are used.

As nodal bases with diagonal norm matrix $\mat{M}$, the nodes of Gauß-Legendre
and Lobatto-Legendre quadrature rules are used.
The new nodal bases represent polynomials of degree $\leq p = 7$ using their
values at the 
\begin{itemize}
  \item roots $\xi_i = \cos \frac{(2i+1) \pi}{2p+2}, \, i = 0, \dots, p$, 
  \item extrema $\xi_i = \cos \frac{i \pi}{p}, \, i = 0, \dots, p$
\end{itemize}
of Chebyshev polynomial $T_{p+1}$ of first kind or the
\begin{itemize}
  \item roots $\xi_i = \cos \frac{(i+1) \pi}{p+2}, \, i = 0, \dots, p$
\end{itemize}
of the Chebyshev polynomial $U_{p+1}$ of second kind. The differentiation and
norm matrices $\mat{D}$, $\mat{M}$ are computed via their representation for
Legendre polynomials and a basis transformation using the associated Vandermonde
matrix, see \ref{sec:appendix-bases}. Multiplication is conducted pointwise
at the corresponding Chebyshev nodes. For these bases, $\mat{M}$ is not diagonal
and multiplication operators $\mat{u}$ are not $\mat{M}$-self-adjoint in general.

Additionally, a modal basis of Legendre polynomials as described in Remark
\ref{rem:bases} is used. An interpolation approach to compute the initial values
for a Legendre basis using the nodes of all
nodal bases presented in Figure \ref{fig:burgers-solution} has been used. There
is no visual difference between results for these different sets of nodes.
In the following, interpolation via Gauß-Legendre nodes has been used.

The momentum and energy of the numerical solutions using the local
Lax-Friedrichs flux are shown in Figure~\ref{fig:burgers-momentum-energy}.
For comparison, the results of \cite{ranocha2016summation} using Gauß-Legendre
and Lobatto-Legendre bases are included in the first rows.
The corresponding numerical solutions are given in \ref{sec:num-sol-Burgers},
Figure~\ref{fig:burgers-solution}, for comparison.

As expected, momentum is conserved for all bases and the discrete energy
(entropy) is constant until $t \approx 0.5$ and decays afterwards, as can be
seen in Figure \ref{fig:burgers-momentum-energy}.

These results are obtained using general SBP CPR methods 
\eqref{eq:inviscid-burgers-CPR-full-correction-general}
with both correction terms for divergence and restriction
\eqref{eq:inviscid-burgers-CPR-correction-terms-general}. Ignoring a non-trivial
correction term for a nodal basis leads to physically useless results, as shown
for example by \cite[Figure 11]{ranocha2016summation}.
Results without the skew-symmetric correction $\vec{c}_\mathrm{div}$ are not plotted
here. Additionally, the correction term $\vec{c}_\mathrm{div}$ using the
$\mat{M}$-adjoint multiplication operator is verified numerically, since using
the simple multiplication as in the previous chapter gives erroneous results,
again not shown here.

Remarkably, the results (not plotted here) using a modal Legendre basis and
either both or no correction term ($\vec{c}_\mathrm{div}$, $\vec{c}_\mathrm{res}$) are
visually indistinguishable. Additionally, using only $\vec{c}_\mathrm{res}$ yields the
same results. Contrary, using only a correction for the divergence results in
varying momentum and physically useless results. Using an exact orthogonal
projection during multiplication seems to be a good idea, but an analytical
investigation of this phenomenon remains an open problem.

\begin{figure}[!hp]
  \centering
  \begin{subfigure}[b]{0.45\textwidth}
    \includegraphics[width=\textwidth]%
      {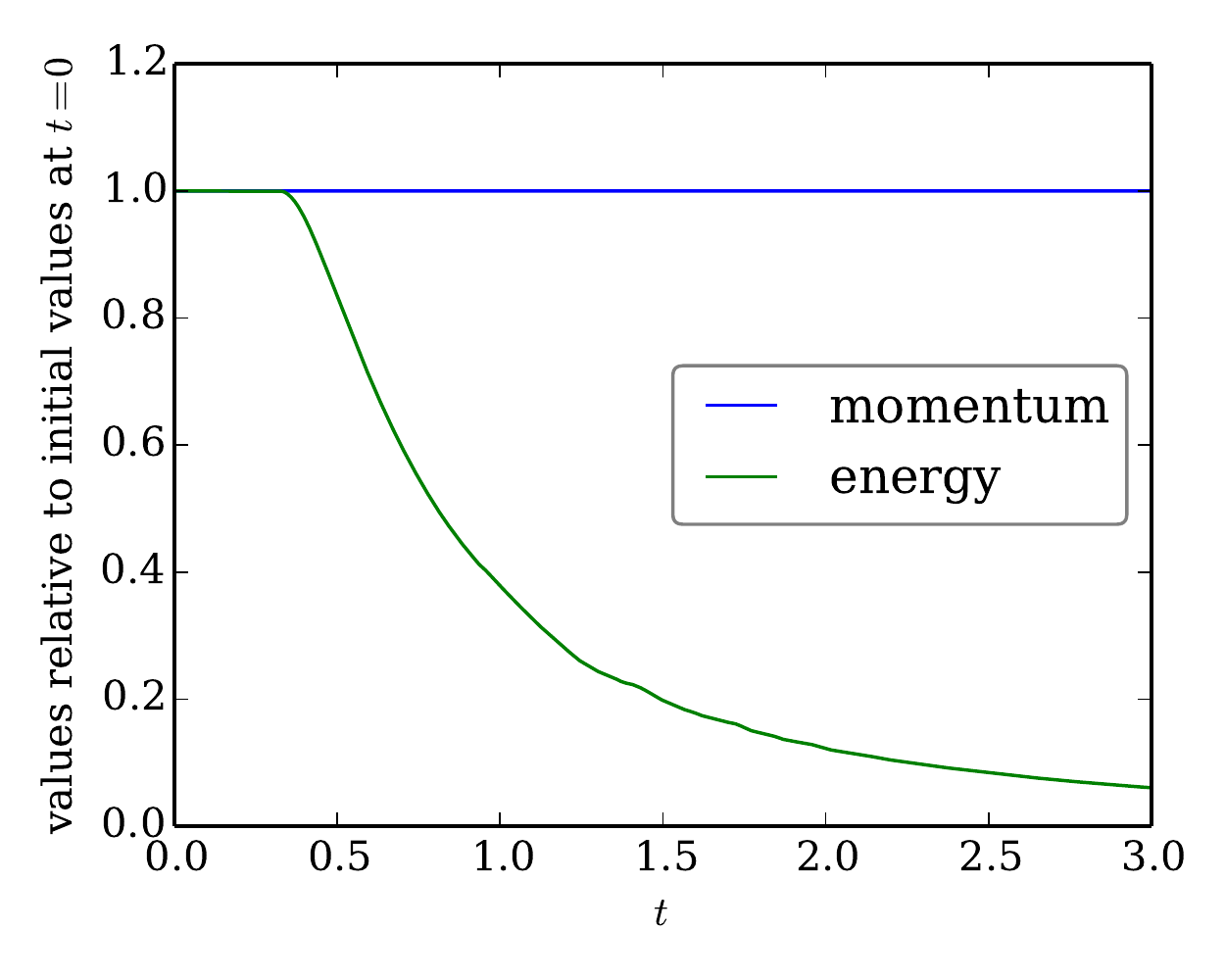}
    \caption{Gauß-Legendre.}
  \end{subfigure}%
  ~
  \begin{subfigure}[b]{0.45\textwidth}
    \includegraphics[width=\textwidth]%
      {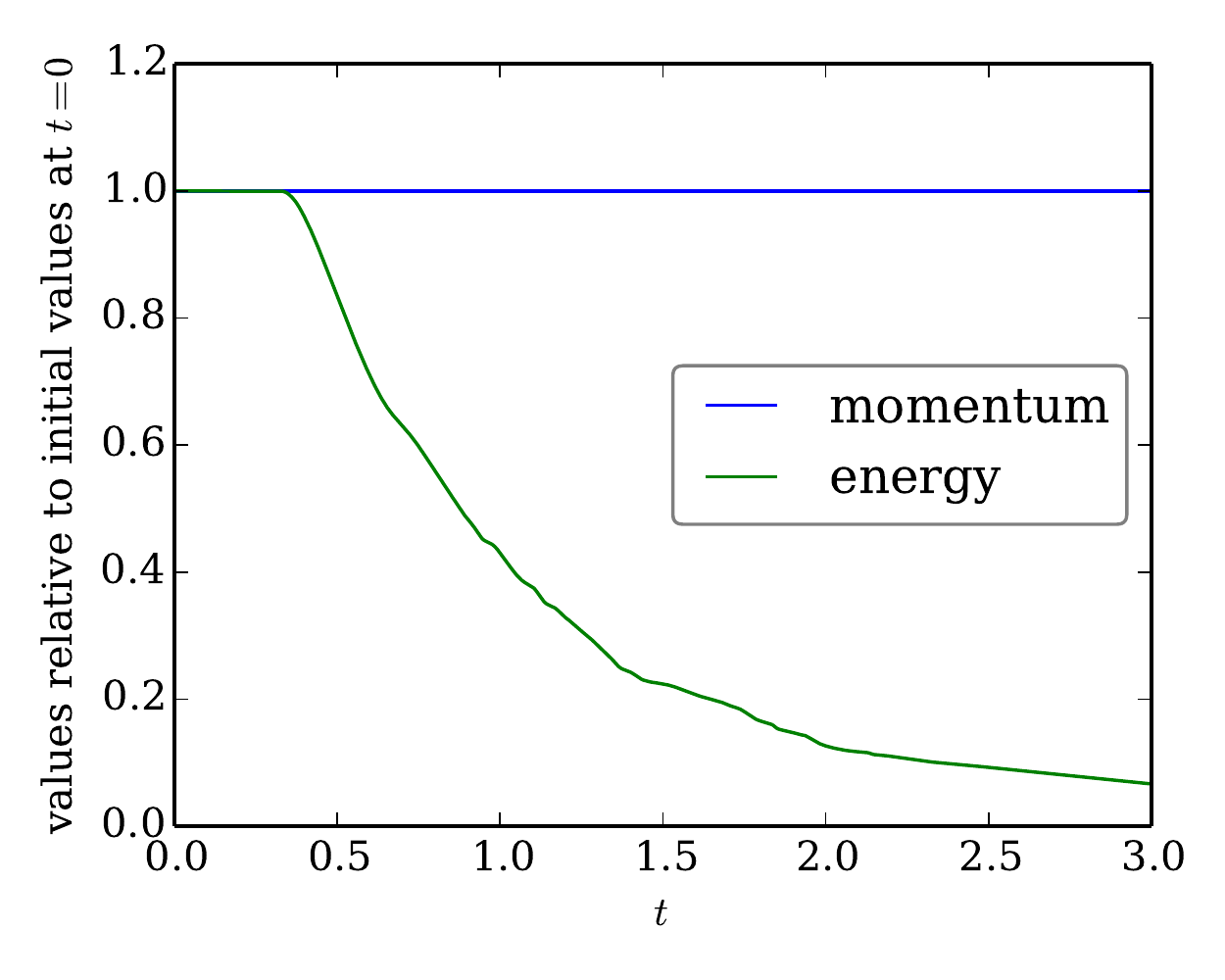}
    \caption{Lobatto-Legendre.}
  \end{subfigure}%
  \\
  \begin{subfigure}[b]{0.45\textwidth}
    \includegraphics[width=\textwidth]%
      {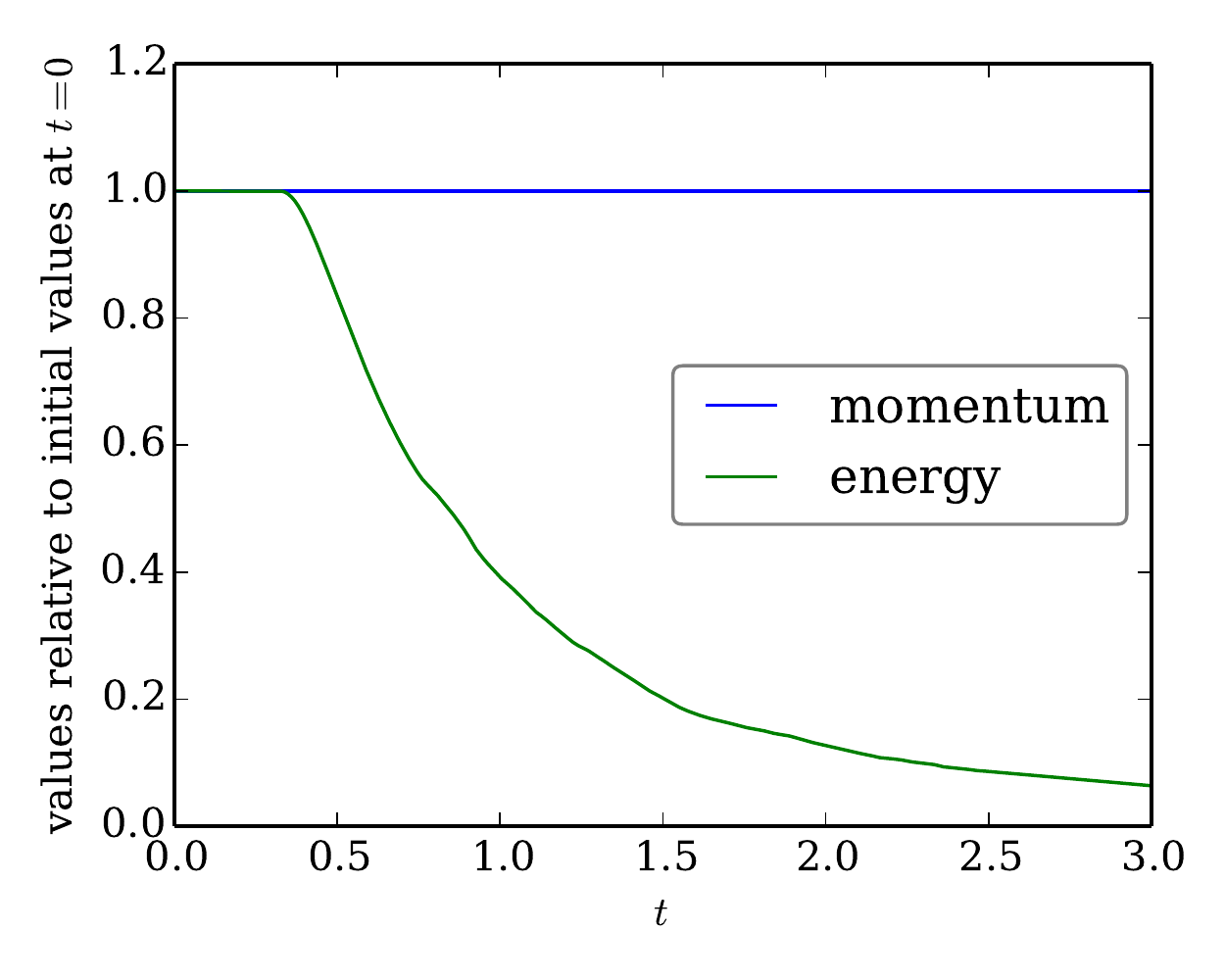}
    \caption{Chebyshev first kind, roots.}
  \end{subfigure}%
  ~
  \begin{subfigure}[b]{0.45\textwidth}
    \includegraphics[width=\textwidth]%
      {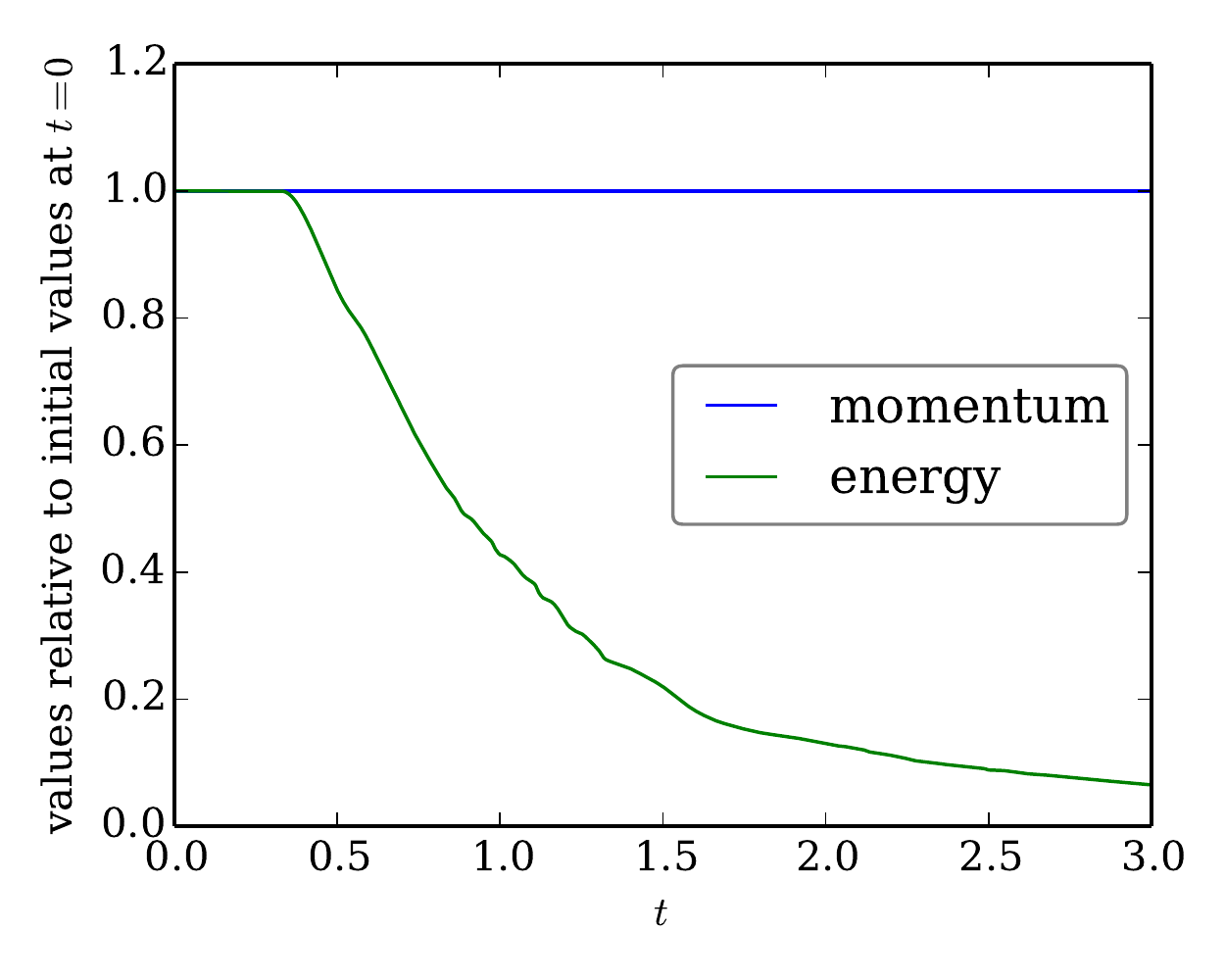}
    \caption{Chebyshev first kind, extrema.}
  \end{subfigure}%
  \\
  \begin{subfigure}[b]{0.45\textwidth}
    \includegraphics[width=\textwidth]%
      {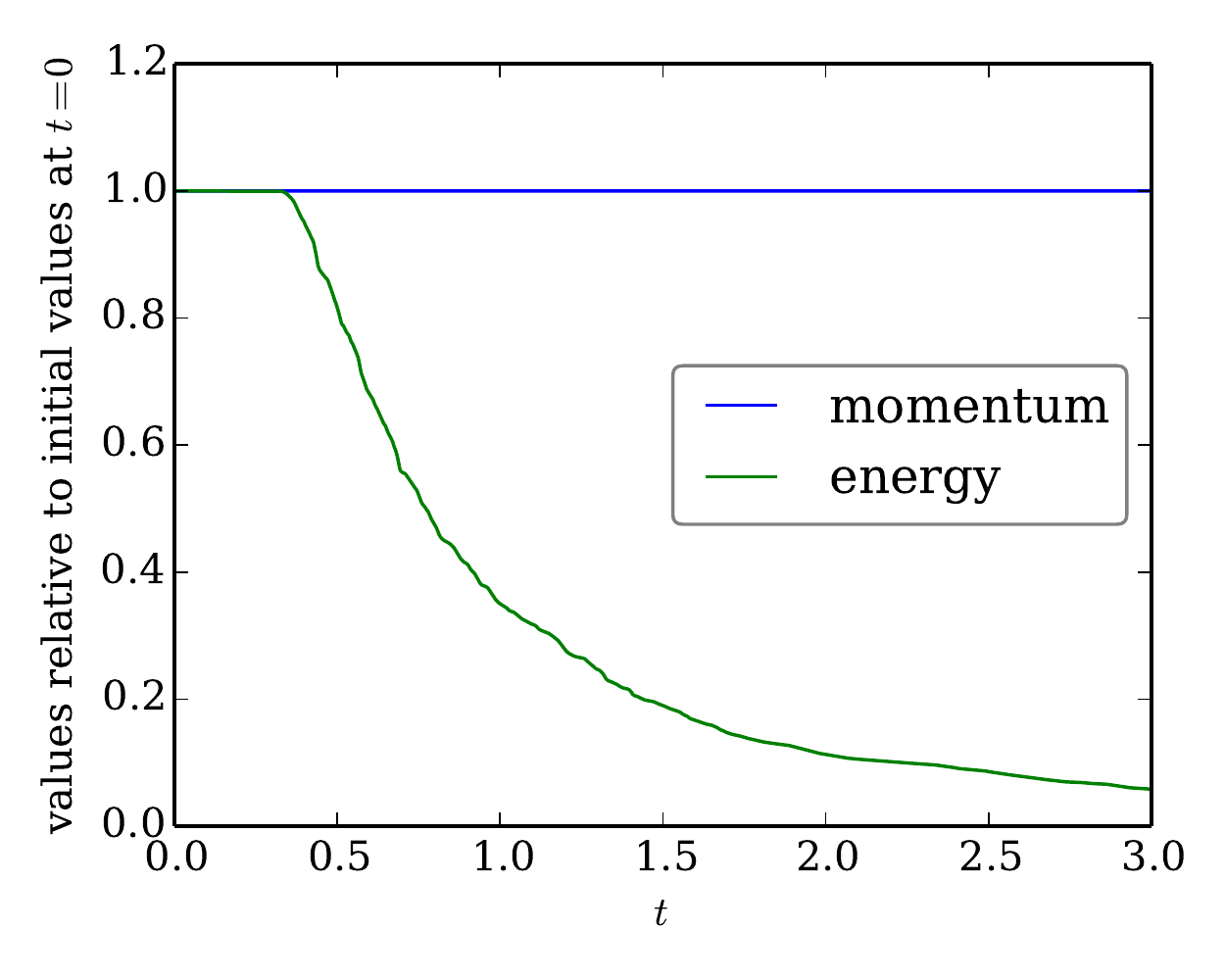}
    \caption{Chebyshev second kind, roots.}
  \end{subfigure}%
  ~
  \begin{subfigure}[b]{0.45\textwidth}
    \includegraphics[width=\textwidth]%
      {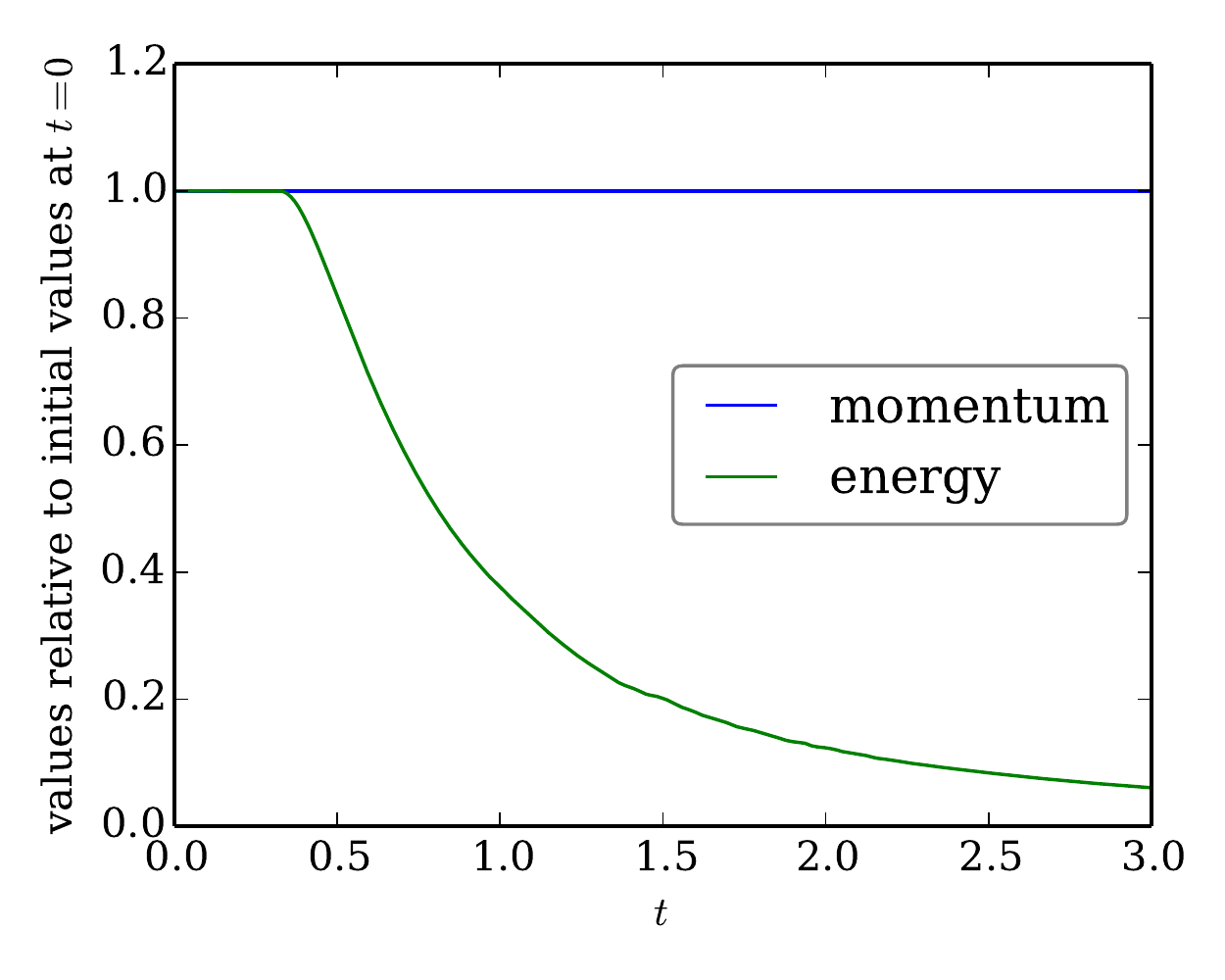}
    \caption{Legendre.}
  \end{subfigure}%
  \caption{Results of the simulations for Burgers' equation
           using general SBP CPR methods with 20 elements, different bases of
           order 7 and local Lax-Friedrichs (LLF) flux.
           Corrections for both divergence and restriction are used.
           Each Figure shows the discrete momentum $\vec{1}^T \mat{M} \vec{u}$
           (blue) and discrete energy $\vec{u}^T \mat{M} \vec{u}$ (green) for
           different bases.
           (For interpretation of the references to colour in this figure legend,
           the reader is referred to the web version of this article.)}
  \label{fig:burgers-momentum-energy}
\end{figure}

\begin{remark}
  These high-order SBP CPR methods should be seen as some entropy stable baseline
  schemes. If calculations involving shocks are performed, these schemes remain
  stable and do not crash, but oscillations occur. Therefore, some form of
  additional shock capturing should be performed, e.g. artificial dissipation or
  modal filtering \cite{ranocha2016enhancing, glaubitz2016enhancing}. However,
  this is not the target of this investigation.
\end{remark}

\subsection{A brief view on a numerical setting}

The analytical setting of section \ref{sec:analytical-setting-in-one-dimension}
is based on a given solution space $X_V$ for the one-dimensional standard
element, since the investigations in this work started from CPR methods,
extending DG methods, which are also described by a fundamental basis.
Contrary, the theory of SBP operators originates in FD methods, classically not
equipped with a solution basis other than the nodal values. Nevertheless,
Gassner \cite{gassner2013skew} adapted the SBP framework to a DGSEM with nodal
Lobatto-Legendre basis and lumped mass matrix. Additionally,
Fernández et al. \cite{fernandez2014generalized} proposed a generalised SBP framework in one
dimension based on nodal values without an analytical basis. Instead, the
operators are required to fulfil the SBP property and some accuracy conditions,
i.e. they should be exact for polynomials up to some degree $p \geq 1$.
These ideas were extended by Hicken et al. \cite{hicken2016multidimensional} to
multi-dimensional operators, focussing on diagonal-norm SBP operators on
simplex elements in two and three dimensions, i.e. triangles and tetrahedra.

These extensions were applied to linear advection with constant velocity
and proved to be conservative and stable in the norm associated with the SBP
operator. Relaxing accuracy conditions potentially results in additional free
parameters, allowing the construction of specialised schemes for different
purposes. As already proved in \cite{hicken2013summation}, SBP operators
are tightly coupled to quadrature rules. Thus, different quadrature rules can
be used to obtain SBP operators and vice versa.

All investigations conducted in the previous chapters and sections directly
extend to these generalised FD SBP operators with diagonal or dense norm,
respectively. Additionally, since these operators are described by the same
matrices used hitherto in the investigations, they can be simply plugged in
the numerical method for the calculations -- up to the last step.

In the analytical setting, the solution is completely determined by the given
coefficients with regard to the chosen basis, i.e. sub-cell resolution of
arbitrary accuracy is given. Especially, the solution can be plotted exactly
as it is used in the computations. Contrary, using only nodal values at a
given set of points without an interpretation as coefficients of a known basis,
only these point values can be plotted as output seriously. Performing any
interpolation would be a guess, but can in general not describe the solution
accurately. From the authors' point of view, this is a drawback of
the numerical setting without a basis as foundation, since high-order
methods can generate highly accurate approximations of smooth solutions with
few degrees of freedom, see also \cite[p. 78, paragraph 2]{kopriva2009implementing}.
However, if the resolution is good enough, a piecewise linear approximation of
nodal values used in most plotting software can be sufficient. Thus, depending on
the applications, this can also be considered a minor drawback.

The inability to describe a modal basis does not seem to be equally unfavourable,
since computing a correct orthogonal projection for division is not a straightforward
task and nodal methods are much more efficient regarding evaluation times for
nonlinear operations.

A solution of the interpolation problem would be to construct a basis describing
a given SBP operator. For example, Gassner \cite{gassner2013skew} constructed a basis
for a specially chosen FD SBP operator. However, there does not seem to be a
straightforward way to construct such a basis in general.

\section{Curvilinear coordinates}
\label{sec:curvilinear}

For real world applications, simple linear coordinate transformations from the
physical element to the reference element might not be enough. Therefore,
curvilinear coordinate transformations have to be applied. However,
Svärd \cite{svard2004coordinate} investigated such transformations in the setting
of finite difference SBP operators. As a result, only diagonal norm SBP operators
have been used for curvilinear coordinates. 
Nevertheless, this sections investigates curvilinear coordinates for the linear
advection equation with constant coefficients in the setting of SBP CPR methods,
also with dense norms.

\subsection{Linear advection}

The linear advection equation 
\begin{equation}
\label{eq:lin-adv}
  \partial_t u + \partial_x u = 0
\end{equation}
with appropriate initial and boundary conditions is used as test problem.
Integrating over an interval $\Omega$ yields the integral form
\begin{equation}
  \od{}{t} \int_\Omega u
  =
  \int_\Omega \partial_t u
  =
  - \int_\Omega \partial_x u
  =
  - u \big|_{\partial \Omega}.
\end{equation}
Similarly, the energy obeys for sufficiently smooth solutions
\begin{equation}
  \frac{1}{2} \od{}{t} \norm{u}_\Omega^2
  =
  \int_\Omega u \partial_t u
  =
  - \int_\Omega u \partial_x u
  =
  - \frac{1}{2} u^2 \big|_{\partial \Omega}.
\end{equation}
If a coordinate mapping $x \mapsto \xi$ from the physical element $\Omega$
to a reference element $\hat\Omega$ is used, the resulting transport equation is
\begin{equation}
  \partial_t u + \partial_x \xi \, \partial_\xi u = 0.
\end{equation}
Using $J = \partial_\xi x = (\partial_x \xi)^{-1}$, this can be written as
\begin{equation}
  J \partial_t u + \partial_\xi u = 0.
\end{equation}
Therefore, integration over the reference element $\hat\Omega$ yields

\begin{equation}
  \od{}{t} \int_\Omega u \dif x
  =
  \od{}{t} \int_{\hat\Omega} J u \dif \xi
  =
  \int_{\hat\Omega} J \partial_t u \dif \xi
  =
  - \int_{\hat\Omega} \partial_\xi u \dif \xi
  =
  - u \big|_{\partial \Omega}.
\end{equation}
Similarly,
\begin{equation}
  \frac{1}{2} \od{}{t} \norm{u}_\Omega^2
  =
  \frac{1}{2} \od{}{t} \int_{\hat\Omega} J u^2 \dif \xi
  =
  \int_{\hat\Omega} J u \partial_t u \dif \xi
  =
  - \int_{\hat\Omega} u \partial_\xi u \dif \xi
  =
  - \frac{1}{2} u^2 \big|_{\partial \Omega}.
\end{equation}
Semidiscretely, the transport equation is approximated on a reference element $\hat\Omega$ by
\begin{equation}
  \mat{J} \partial_t \vec{u}
  + \mat{D} \vec{u}
  + \mat{M}[^{-1}] \mat{R}[^T] \mat{B} \left( \vecfnum - \mat{R} \vec{u} \right)
  = \vec{0}.
\end{equation}
Investigating conservation, the SBP property \eqref{eq:SBP} and exactness of
differentiation for constants, i.e. $\mat{D} \vec{1} = 0$, can be used to get
\begin{equation}
\label{eq:curvilinear-conservation}
\begin{aligned}
  \vec{1}^T \mat{M} \mat{J} \partial_t \vec{u}
  =&
  - \vec{1}^T \mat{M} \mat{D} \vec{u}
  - \vec{1}^T \mat{R}[^T] \mat{B} \left( \vecfnum - \mat{R} \vec{u} \right)
  \\
  =&
  \vec{1}^T \mat{D}[^T] \mat{M} \vec{u}
  - \vec{1}^T \mat{R}[^T] \mat{B} \mat{R} \vec{u}
  - \vec{1}^T \mat{R}[^T] \mat{B} \left( \vecfnum - \mat{R} \vec{u} \right)
  =
  - \vec{1}^T \mat{R}[^T] \mat{B} \vecfnum.
\end{aligned}
\end{equation}
Similarly,
\begin{equation}
\label{eq:curvilinear-energy}
\begin{aligned}
  &
  \vec{u}^T \mat{M} \mat{J} \partial_t \vec{u}
  =
  - \vec{u}^T \mat{M} \mat{D} \vec{u}
  - \vec{u}^T \mat{R}[^T] \mat{B} \left( \vecfnum - \mat{R} \vec{u} \right)
  \\
  =&
  - \frac{1}{2} \vec{u}^T \mat{M} \mat{D} \vec{u}
  + \frac{1}{2} \vec{u}^T \mat{D}[^T] \mat{M} \vec{u}
  - \frac{1}{2} \vec{u}^T \mat{R}[^T] \mat{B} \mat{R} \vec{u}
  - \vec{u}^T \mat{R}[^T] \mat{B} \left( \vecfnum - \mat{R} \vec{u} \right)
  =
  - \vec{u}^T \mat{R}[^T] \mat{B}
    \left( \vecfnum - \frac{1}{2} \mat{R} \vec{u} \right).
\end{aligned}
\end{equation}
To use \eqref{eq:curvilinear-energy} as an energy estimate, $\mat{M} \mat{J}$ has
to be symmetric and positive definite.
There are three cases that can be handled differently:
\begin{enumerate}
  \item
  If the coordinate transformation $x \mapsto \xi$ satisfies $J = \partial_\xi x > 0$,
  a nodal SBP basis with diagonal norm $\mat{M}$ and pointwise multiplication
  operator $\mat{J}$ is used, both $\mat{M}$ and $\mat{J}$ are diagonal matrices
  with positive entries on the diagonal. Therefore, $\mat{M} \mat{J}$ is symmetric
  and positive definite. 
  This is fulfilled inter alia by Gauß-Legendre and Lobatto-Legendre bases.

  \item
  If a modal Legendre basis is used and $\mat{J}$ is given by exact multiplication
  followed by exact $L^2$ projection, 
  \begin{equation}
    \vec{v}^T \mat{M} \mat{J} \vec{u}
    =
    \int v \proj(J u)
    =
    \int v \, J u
    =
    \int \proj(J v) u
    =
    \vec{v}^T \mat{J}[^T] \mat{M} \vec{u},
  \end{equation}
  since the Legendre polynomials are orthogonal. Thus, $\mat{J}$ is
  $\mat{M}$-self-adjoint and $\mat{M} \mat{J} = \mat{J}[^T] \mat{M}$ is symmetric
  and positive definite for $J = \partial_\xi x > 0$.

  \item
  As described in \ref{sec:appendix-bases}, coordinate transformations
  can be used to transform multiplication operators $\mat{J}$ from one basis to
  another. If $\vec{u}, \vec{\hat u}$ are vectors represented in different bases
  and $\mat{V}$ is the corresponding transformation matrix, i.e.
  $\vec{u} = \mat{V} \vec{\hat u}$, the matrices $\mat{J}$ and $\mat{\hat J}$ are
  related via $\mat{\hat J} = \mat{V}[^{-1}] \mat{J} \mat{V}$.
  This can be used to transform the multiplication operator $\mat{J}$ and
  the mass matrix $\mat{M}$ from a bases in which $\mat{M} \mat{J}$ is symmetric
  and positive definite [e.g. Gauß-Legendre nodes with diagonal multiplication
  matrix] to another one [e.g. dense norm Chebyshev bases as in section 
  \ref{sec:numerical-tests-Burgers}].
\end{enumerate}

\subsection{Numerical results}

The initial condition $u_0(x) = \exp\left( -20 x^2 \right)$ of the linear advection 
equation \eqref{eq:lin-adv} has been transported in the time interval $[0,4]$
using the classical fourth order Runge-Kutta method.
The domain $[-1,1]$ is divided into $5$ elements using polynomials of degree
$\leq p = 9$ on each element. Here, three different grids have been considered.
\begin{itemize}
  \item 
  Grid 1: Uniform widths
  $\Delta x_i = \Delta x_{i-1}$, $i \in \set{2,3,4,5}$.
  
  \item
  Grid 2: Alternating widths
  $\Delta x_i = 10 \Delta x_{i-1}$ for $i \in \set{3,5}$ and
  $\Delta x_i = \frac{1}{10} \Delta x_{i-1}$ for $i \in \set{2,4}$.
  
  \item
  Grid 3: Geometrically increasing widths
  $\Delta x_i = \frac{3}{2} \Delta x_{i-1}$, $i \in \set{2,3,4,5}$.
\end{itemize}

As polynomial bases, nodal bases using Gauß-Legendre and Lobatto-Legendre nodes 
as well as the roots of the Chebyshev polynomials $U_{p+1}$ of second kind
have been chosen.
As numerical flux, the central flux $\fnum(u_-, u_+) = \frac{u_- + u_+}{2}$
is used.

Two kinds of (local) mappings from the reference element $[-1,1]$ to each of the
five physical elements $[x_\mathrm{min}, x_\mathrm{max}]$
have been chosen:
\begin{itemize}
  \item 
  The standard linear mapping
  $[-1,1] \ni \xi \mapsto \frac{x_\mathrm{max}+x_\mathrm{min}}{2} + \xi \frac{x_\mathrm{max}-x_\mathrm{min}}{2}$.
  
  \item
  The quadratic mapping
  $[-1,1] \ni \xi \mapsto \frac{x_\mathrm{max}-x_\mathrm{min}}{8} (\xi+2)^2 + x_\mathrm{min} - \frac{x_\mathrm{max}-x_\mathrm{min}}{8}$.
\end{itemize}
Computing $\partial_\xi x$ via differentiation of the interpolation polynomial,
two versions for the computation of $\mat{J}$ have been conducted:
\begin{itemize}
  \item 
  As diagonal multiplication operator $\diag{ \partial_\xi x \big|_{\xi=\xi_i} }$.
  
  \item
  As $\mat{V} \mat{\tilde J} \mat{V}[^{-1}]$, where $\mat{\tilde J}$ is the
  diagonal Jacobian in a basis using Gauß-Legendre nodes.
\end{itemize}

The numerical results for the linear mapping and all choices considered here
are visually indistinguishable from the initial data which are also the solution
at $t = 4$. Therefore, these results are not shown here.

However, if the quadratic mapping is used for Chebyshev nodes and $\mat{J}$ is
computed as diagonal multiplication matrix, the numerical solution blows up.
If the stable choice of computing $\mat{J}$ via Gauß-Legendre nodes is
used, the solution is visually the same as the one for the linear mapping, as
can be seen in Figure \ref{fig:lin-adv-quadratic} [cf. case 3. in the previous
section].

Similarly, if the diagonal multiplication matrix $\mat{J}$ is used for
Lobatto-Legendre nodes, the solutions is stable [cf. case 1. in the previous
section]. However, since the lumped mass matrix is not exact as for Gauß-Legendre
nodes, computing $\mat{J}$ via Gauß-Legendre nodes results in a blow up of the
numerical solution.

The results are qualitatively independent on the choice of the grid for the
three different possibilities described above: If the method is stable, the
results are visually indistinguishable from the exact solution. Otherwise,
the numerical solutions blows up.

\begin{figure}[!hp]
  \centering
  \begin{subfigure}[b]{0.42\textwidth}
    \includegraphics[width=\textwidth]%
      {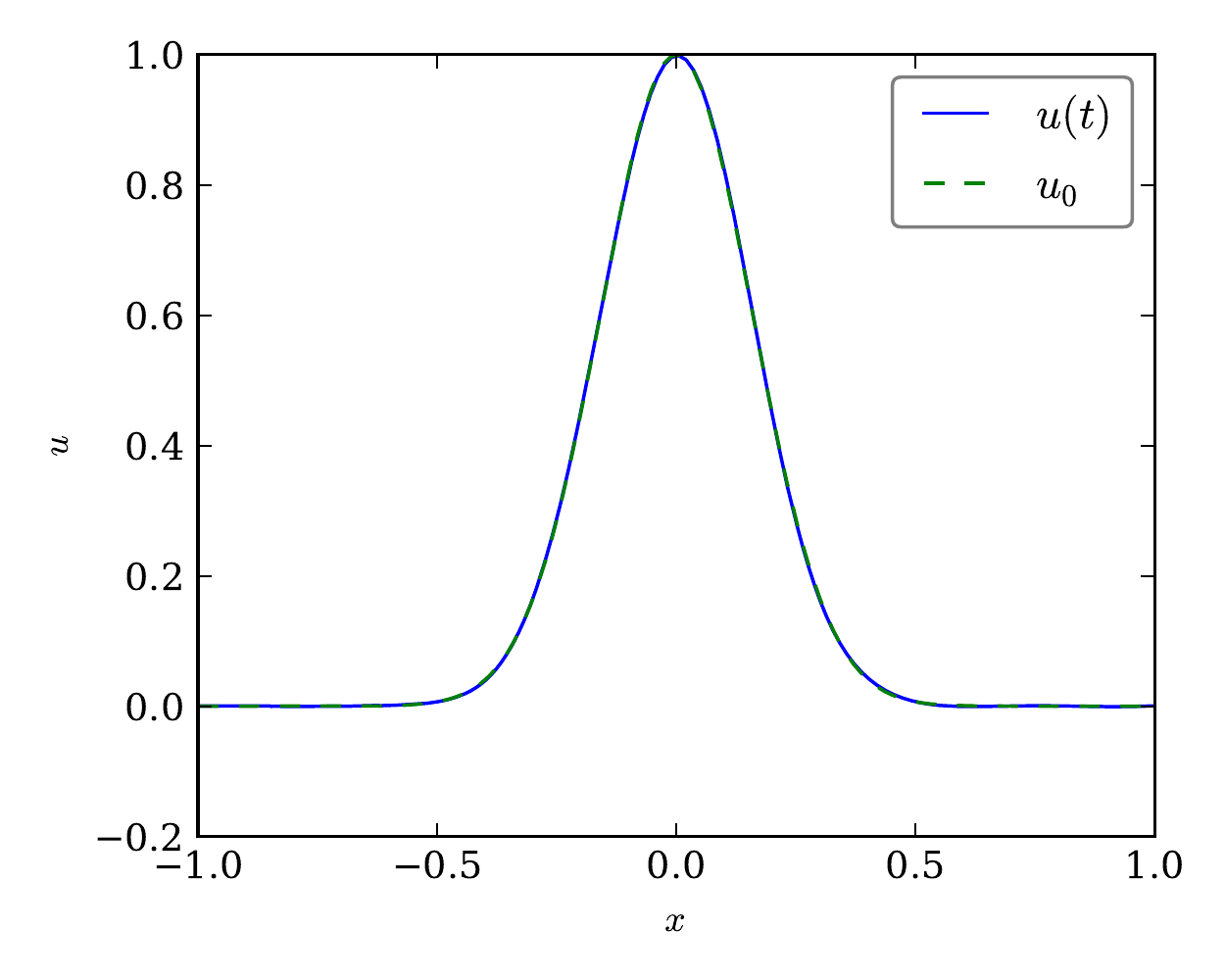}
    \caption{Chebyshev second kind (roots),
             $\mat{J}$ computed via Gauß-Legendre nodes.}
  \end{subfigure}%
  ~
  \begin{subfigure}[b]{0.42\textwidth}
    \includegraphics[width=\textwidth]%
      {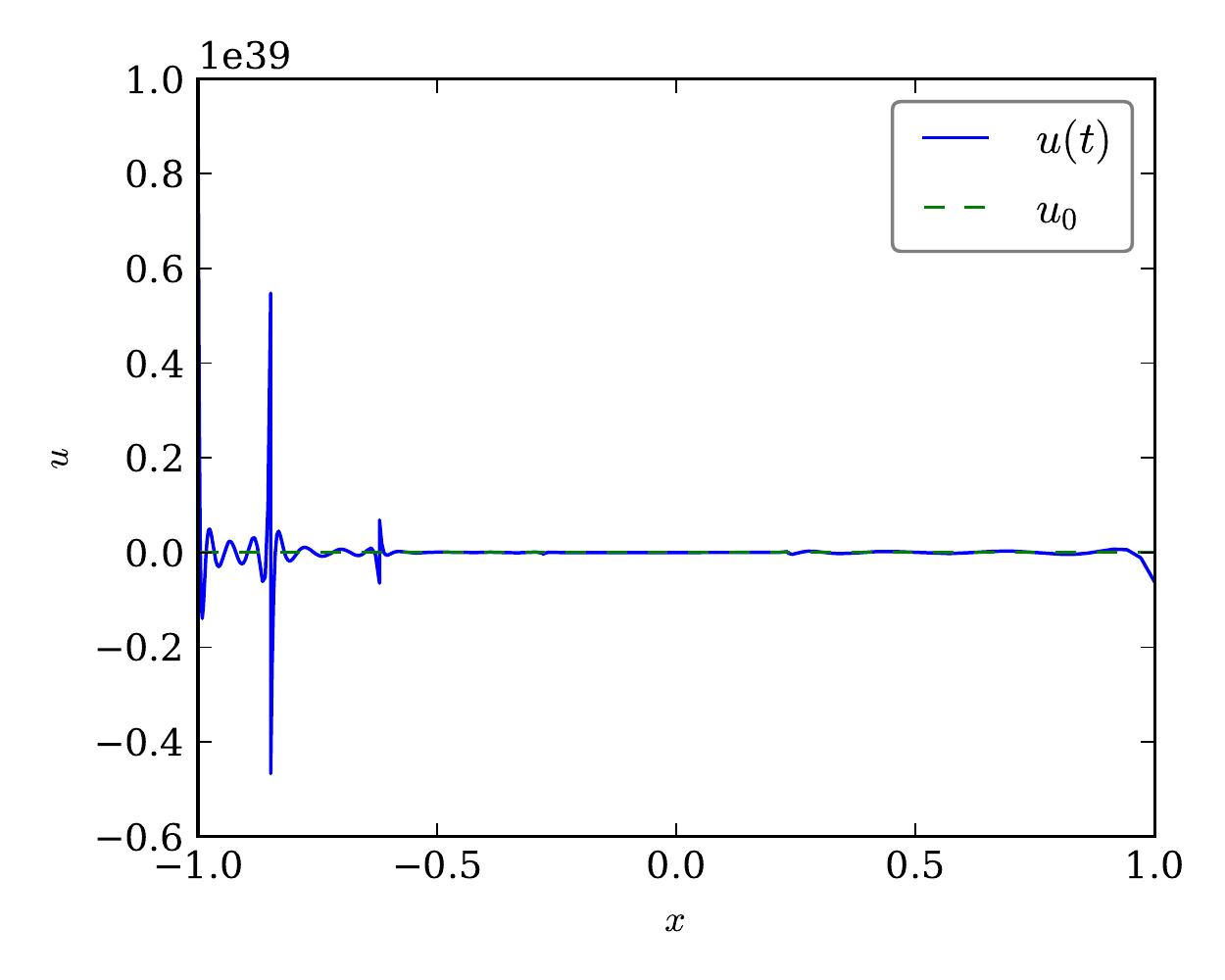}
    \caption{Chebyshev second kind (roots),
             $\mat{J}$ as diagonal multiplication matrix.}
  \end{subfigure}%
  \\
  \begin{subfigure}[b]{0.42\textwidth}
    \includegraphics[width=\textwidth]%
      {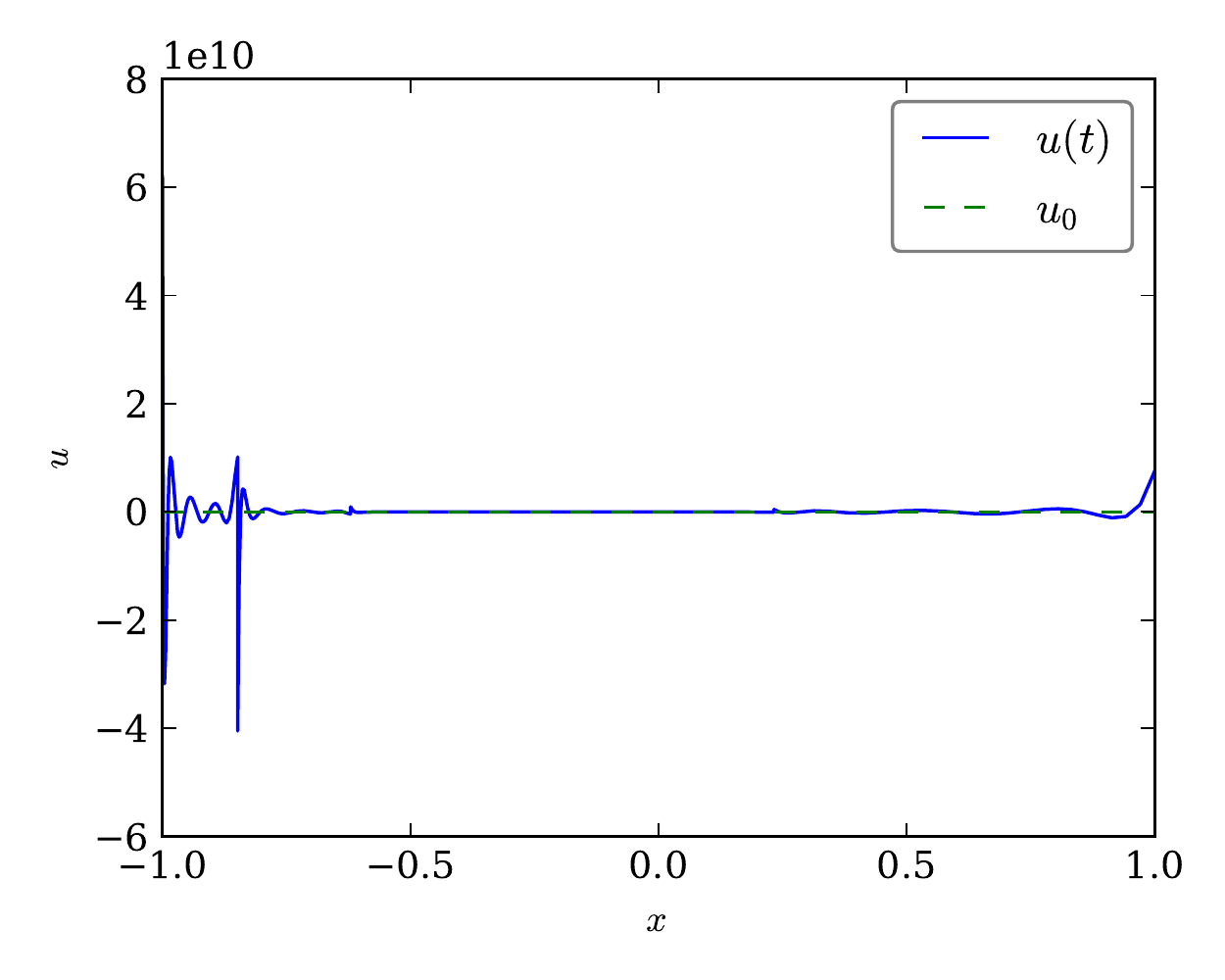}
    \caption{Lobatto-Legendre nodes,
             $\mat{J}$ computed via Gauß-Legendre nodes.}
  \end{subfigure}%
  ~
  \begin{subfigure}[b]{0.42\textwidth}
    \includegraphics[width=\textwidth]%
      {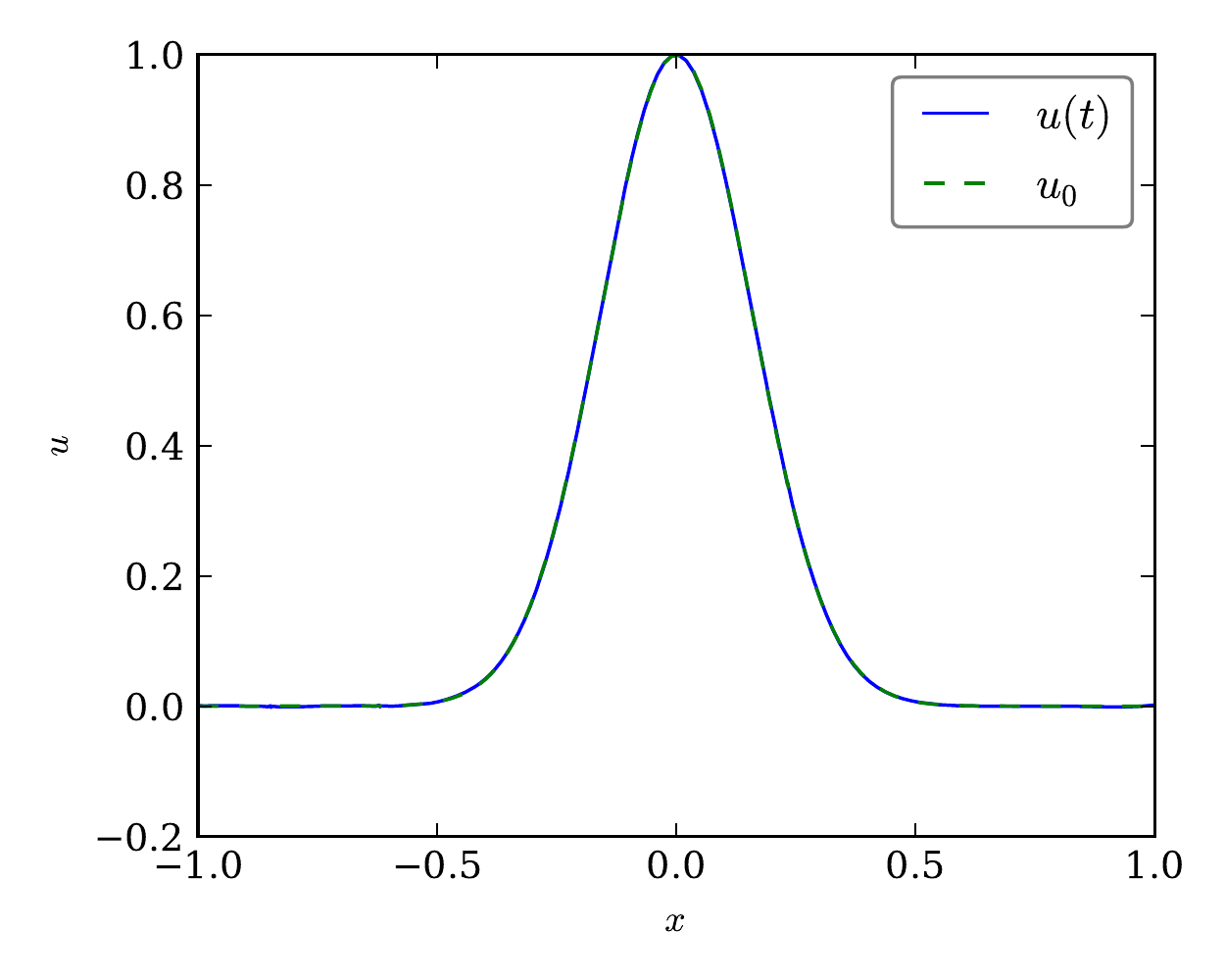}
    \caption{Lobatto-Legendre nodes,
             $\mat{J}$ as diagonal multiplication matrix.}
  \end{subfigure}%
  \\
  \begin{subfigure}[b]{0.42\textwidth}
    \includegraphics[width=\textwidth]%
      {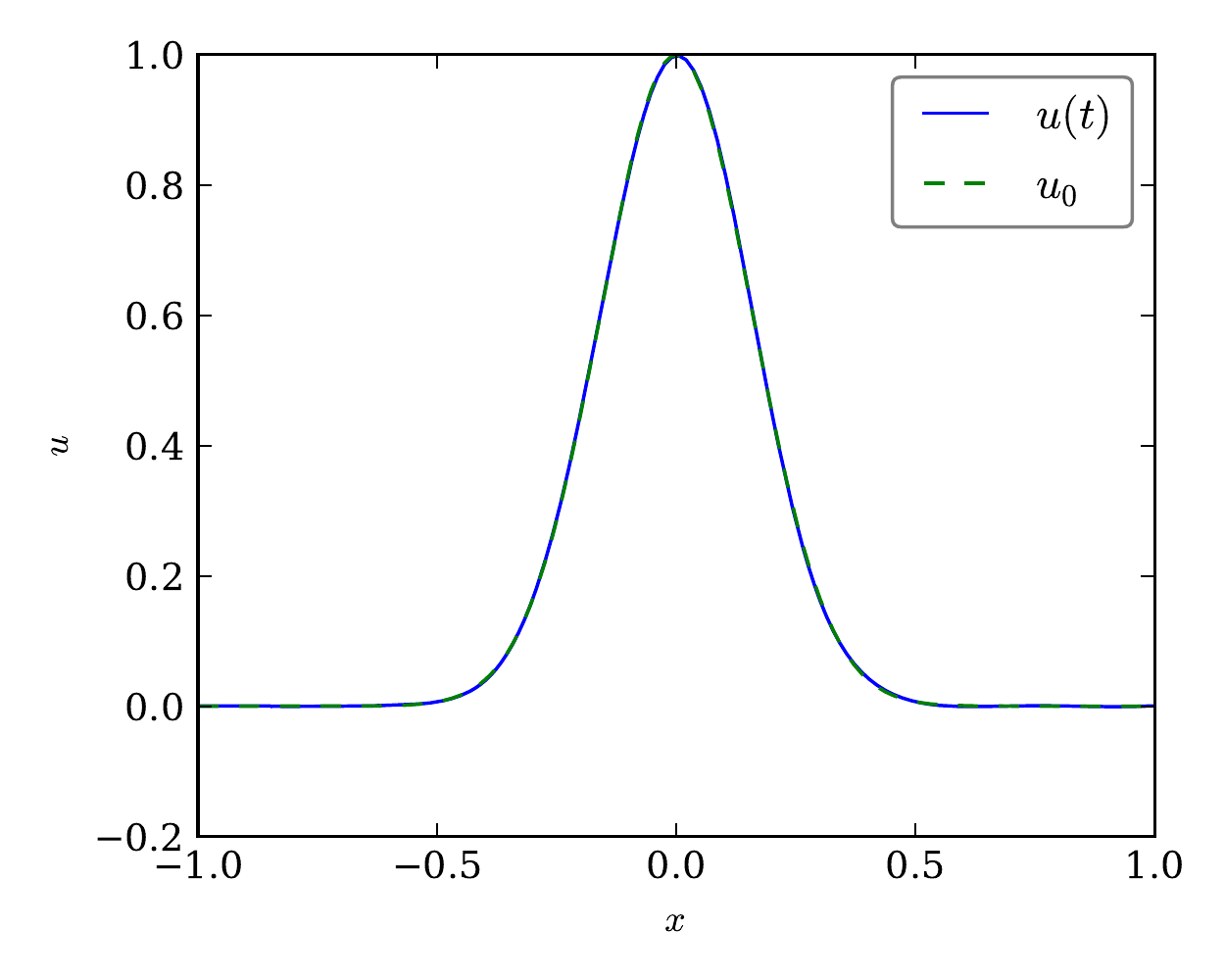}
    \caption{Gauß-Legendre nodes,
             $\mat{J}$ as diagonal multiplication matrix.}
  \end{subfigure}%
  \caption{Results of the simulations for linear advection using
           $5$ elements on grid 3 (geometrically increasing widths)
           with polynomials of degree $\leq p = 9$, the central flux, and the
           quadratic mapping.
           Each Figure shows the values of $u(t=4)$ (blue line) and $u(0) = u_0$
           (green, dashed).
           (For interpretation of the references to colour in this figure legend,
           the reader is referred to the web version of this article.)}
  \label{fig:lin-adv-quadratic}
\end{figure}

\subsection{Multiple dimensions}

In order to apply the methods of the previous sections to multi-dimensional problems,
it is common to use tensor product bases and rectangular grids. Only curvilinear
coordinates deserve a special treatment in this case, since all other desired
properties extend naturally to several space dimensions. However, genuinely
multi-dimensional SBP operators can be constructed as well. Here, the extension
of the analytical setting of section \ref{sec:generalisation} is described, similarly
to the numerical setting of \cite{hicken2016multidimensional}, see also
\cite[section 6]{ranocha2016sbp}.

As in section \ref{sec:generalisation}, there is a finite-dimensional (real)
Hilbert space $X_V$ of functions on the $d$ dimensional reference (volume)
element $\Omega$, equipped with a basis $\mathcal{B}_V$. With regard to this
basis, the scalar product is represented by the mass matrix $\mat{M}$ and
approximates the $L^2$ norm on $X_V$.
Moreover, there are $d$ derivative operators $\mat{D_i}, i \in \set{1, \dots, d}$,
denoting the partial derivative in the $i$-th coordinate direction.

Furthermore, the functions on the $d-1$ dimensional boundary $\partial\Omega$
are members of the Hilbert space $X_B$ with basis $\mathcal{B}_B$. The associated
scalar product is given by the matrix $\mat{B}$, approximating the $L^2$ scalar
product via
\begin{equation}
  \vec{u_B}^T \mat{B} \vec{v_B}
  =
  \left\langle \vec{u_B}, \vec{v_B} \right\rangle_B
  \approx 
  \left\langle u_B, v_B \right\rangle_{L^2(\partial\Omega)}
  =
  \int_{\partial\Omega} u_B v_B.
\end{equation}
As in one dimension, there is a restriction operator $\mat{R}$ mapping $X_V$ to
$X_B$. Finally, there are $d$ operators $\mat{N_i}, i \in \set{1,\dots,d},$ on
$X_B$ performing multiplication with the $i$-th component of the outer unit normal $\nu$
at $\partial\Omega$. Together, these operators approximate 
\begin{equation}
  \vec{u}^T \mat{R}[^T] \mat{B} \mat{N_i} \mat{R} \vec{v}
  \approx
  \int_{\partial\Omega} u v \, \nu_i.
\end{equation}
In the end, the SBP property in multiple dimensions can be formulated as
\begin{equation}
\label{eq:SBP-multi-dim}
  \mat{M} \mat{D_i} + \mat{D_i}[^T] \mat{M}
  =
  \mat{R}[^T] \mat{B} \mat{N_i} \mat{R},
\end{equation}
mimicking the divergence theorem
\begin{equation}
  \int_\Omega u (\partial_i v) + \int_\Omega (\partial_i u) v
  \approx
  \vec{u}^T \mat{M} \mat{D_i} \vec{v} + \vec{u}^T \mat{D_i}[^T] \mat{M} \vec{v}
  =
  \vec{u}^T \mat{R}[^T] \mat{B} \mat{N_i} \mat{R} \vec{v}
  \approx
  \int_{\partial\Omega} u v \, \nu_i.
\end{equation}
Of course, tensor product bases formed by one-dimensional SBP bases fulfil the 
requirements for multi-dimensional SBP bases.

\begin{remark}
  For curvilinear grids in multiple dimensions, the metric identities are crucial
  for free stream preservation, conservation and stability, as described inter
  alia in \cite{kopriva2006metric}. An application using split-form discontinuous
  Galerkin methods for the shallow water equations in two-dimensions on curvilinear
  grids is presented in \cite{wintermeyer2015entropy}.
\end{remark}

\section{Summary and discussion}
\label{sec:summary}

In this work, an extended analytical framework for SBP methods has been
proposed. Using the results of \cite{ranocha2016summation}, the linearly stable
CPR methods of \cite{vincent2011newclass, vincent2015extended} and the DGSEM of
\cite{gassner2013skew} are embedded in this framework.

Additionally, new forms of correction terms for nonlinear conservation laws are
developed, using the inviscid Burgers' equation as an example. These correction
terms for both divergence and restriction to the boundary extend the skew-symmetric
form of conservation laws used in traditional FD SBP methods and the DGSEM based
on Lobatto nodes \cite{fisher2013discretely, gassner2013skew}.

For the first time, these new corrections allow for both modal and nodal SBP bases
without any further conditions on the norm (e.g. diagonal) or the presence of
nodes at the boundary. Using the SBP property, both conservation and stability
in a discrete norm adapted to the chosen bases are proved. These results
extend directly to traditional SBP methods lacking the foundation of an
analytical basis, since only structural properties of the representations in a
given basis are used.

Moreover, stability for curvilinear grids and dense norms is obtained by using a
suitable way to compute the Jacobian. Thus, complications that have been known
in the FD framework of SBP methods can be circumvented for SBP CPR methods.

A straightforward extension of the analytical setting to multiple dimensions is described
in \cite{ranocha2016sbp}. Similar to the numerical setting of
\cite{fernandez2014generalized, hicken2016multidimensional}, this genuinely
multi-dimensional formulation allows inter alia simplex elements and does not
rely on a tensor product extension. Of course, the standard multi-dimensional
setting using tensor products is embedded therein.

Further research includes fully discrete schemes and other examples for nonlinear 
systems of conservation laws.

\appendix

\section{Some bases}
\label{sec:appendix-bases}

The analytical setting described in section \ref{sec:analytical-setting-in-one-dimension}
uses finite dimensional Hilbert spaces to represent numerical solutions.
In all cases considered in this article, the space of polynomials of degree
$\leq p$ has been used. However, for concrete computations, a basis has to be
selected. If the derivative $\mat{D}$, restriction $\mat{R}$, and mass matrix
$\mat{M}$ are exact for polynomials of degree $\leq p$, the SBP property \eqref{eq:SBP}
is fulfilled, since integration by parts can be applied, see also
\cite{kopriva2010quadrature,hicken2013summation}:
\begin{equation*}
  \vec{u}^T \mat{M} \mat{D} \vec{v}
  + \vec{u}^T \mat{D}[^T] \mat{M} \vec{v}
  =
  \int_{-1}^1 u \, (\partial_x v) \dif x
  + \int_{-1}^1 (\partial_x u) \, v \dif x
  =
  \eval[2]{u \, v}_{-1}^{1}
  =
  \vec{u}^T \mat{R}[^T] \mat{B} \mat{R} \vec{v}.
\end{equation*}
Furthermore, the matrix representations of linear operators ($\mat{R}, \mat{D}$)
and bilinear forms ($\mat{M}$) can be computed in one basis and then transformed
via the standard transformation rules to another basis (at least in theory and 
for small polynomial degrees, since the condition numbers might increase drastically
for higher polynomial degrees).

To compute the matrices $\mat{M}, \mat{D}$ for the nodal bases using Chebyshev
points, the associated matrices in a modal Legendre basis are used.
The coordinate transformation from a nodal basis with nodes $\xi_0, \dots, \xi_p$
to a modal basis of Legendre polynomials $\phi_0, \dots, \phi_p$ of degree
$\leq p$ is given by the Vandermonde matrix $\mat{V}$ with
$V_{i,j} = \phi_j(\xi_i)$. Writing vectors and matrices with regard to the
modal basis with $\hat{\cdot}$, the transformation is
$\mat{V} \vec{\hat{u}} = \vec{u}$. Thus, operators like the derivative are
transformed as $\mat{\hat{D}} = \mat{V}[^{-1}] \mat{D} \mat{V}$ and matrices
associated with a scalar product like $\mat{M}$ as
$\mat{\hat{M}} = \mat{V}[^T] \mat{M} \mat{V}$.

The modal matrices are
\begin{equation}
  \mat{\hat{M}} =
  \begin{pmatrix}
     2      
  \\ &  \frac{2}{3}
  \\ & & \ddots
  \\ & & & \frac{2}{2p+1}
  \end{pmatrix}
  ,\quad
  \mat{\hat{D}} =
  \begin{pmatrix}
     0      &  1     &  0     &     1  &  0     &  \dots
  \\ 0      &  0     &  3     &     0  &  3     &  \dots
  \\ 0      &  0     &  0     &     5  &  0     &  \dots
  \\ 0      &  0     &  0     &     0  &  7     &  \dots
  \\ \vdots & \vdots & \vdots & \vdots & \vdots & \ddots
  \end{pmatrix}.
\end{equation}

Using $p = 2$ as an example, the nodal bases with dense norm are given by the
following matrices.
\begin{itemize}
  \item
  The roots of the Chebyshev polynomials of first kind are
  $\xi_i = \cos\left( \frac{2i+1}{2p+2}\pi \right)$, for $i \in \set{p,p-1,\dots,0}$.
  The Vandermonde matrix is
  \begin{equation}
  \renewcommand{\arraystretch}{1.5}
    \mat{V} =
    \begin{pmatrix}
      1  &  -\frac{\sqrt{3}}{2}  & \frac{5}{8} \\
      1  &  0                    & -\frac{1}{2} \\
      1  &  \frac{\sqrt{3}}{2}   & \frac{5}{8} \\
    \end{pmatrix}.
  \end{equation}
  Calculating the mass matrix as $\mat{M} = \mat{V}^{-T} \mat{\hat{M}} \mat{V}$
  results in
  \begin{equation}
  \renewcommand{\arraystretch}{1.5}
    \mat{M} =
    \begin{pmatrix}
      \frac{2}{5}   &  \frac{4}{45}   &  -\frac{2}{45} \\
      \frac{4}{45}  &  \frac{14}{15}  &  \frac{4}{45} \\
      -\frac{2}{45} &  \frac{4}{45}   &  \frac{2}{5} \\
    \end{pmatrix}.
  \end{equation}
  The restriction (interpolation to the boundary) and boundary matrices used are
  \begin{equation}
  \renewcommand{\arraystretch}{1.5}
    \mat{R} = 
    \begin{pmatrix}
      \frac{2 + \sqrt{3}}{3}  &  -\frac{1}{3}  &  \frac{2 - \sqrt{3}}{3} \\
      \frac{2 - \sqrt{3}}{3}  &  -\frac{1}{3}  &  \frac{2 + \sqrt{3}}{3} \\
    \end{pmatrix},
    \quad
    \mat{B} = 
    \begin{pmatrix}
      -1  &  0 \\
      0   &  1 \\
    \end{pmatrix}.
  \end{equation}
  Computing the derivative matrix via $\mat{D} = \mat{V} \mat{\hat{D}} \mat{V}^{-1}$
  yields
  \begin{equation}
  \renewcommand{\arraystretch}{1.5}
    \mat{D} = 
    \begin{pmatrix}
      -\sqrt{3}            &  \frac{4 \sqrt{3}}{3}   &  -\frac{\sqrt{3}}{3} \\
      -\frac{\sqrt{3}}{3}  &  0                      &  \frac{\sqrt{3}}{3} \\
      \frac{\sqrt{3}}{3}   &  -\frac{4 \sqrt{3}}{3}  &  \sqrt{3} \\
    \end{pmatrix}.
  \end{equation}

  \item
  The extrema of the Chebyshev polynomials of first kind are 
  $\xi_i = \cos\left( \frac{i}{p}\pi \right)$, for $i \in \set{p,p-1,\dots,0}$.
  Thus, the matrices are
  \begin{gather}
  \renewcommand{\arraystretch}{1.5}
    \mat{V} = 
    \begin{pmatrix}
      1 & -1 & 1 \\
      1 & 0 & -\frac{1}{2} \\
      1 & 1 & 1
    \end{pmatrix},
    \quad
    \mat{M} = 
    \begin{pmatrix}
      \frac{4}{15} & \frac{2}{15} & -\frac{1}{15} \\
      \frac{2}{15} & \frac{16}{15} & \frac{2}{15} \\
      -\frac{1}{15} & \frac{2}{15} & \frac{4}{15}
    \end{pmatrix},
    \quad
    \mat{R} = 
    \begin{pmatrix}
      1 & 0 & 0 \\
      0 & 0 & 1
    \end{pmatrix},
    \quad
    \mat{D} = 
    \begin{pmatrix}
      -\frac{3}{2} & 2 & -\frac{1}{2} \\
      -\frac{1}{2} & 0 & \frac{1}{2} \\
      \frac{1}{2} & -2 & \frac{3}{3}
    \end{pmatrix}.
  \end{gather}

  \item
  Finally, the roots of the Chebyshev polynomials of second kind are
  $\xi_i = \cos\left( \frac{i+1}{p+2}\pi \right)$, where again
  $i \in \set{p,p-1,\dots,0}$. Therefore, the matrices are
  \begin{gather}
  \renewcommand{\arraystretch}{1.5}
    \mat{V} = 
    \begin{pmatrix}
      1 & -\frac{\sqrt{2}}{2} & \frac{1}{4} \\
      1 & 0 & -\frac{1}{2} \\
      1 & \frac{\sqrt{2}}{2} & \frac{1}{4}
    \end{pmatrix},
    \quad
    \mat{M} = 
    \begin{pmatrix}
      \frac{11}{15} & -\frac{2}{15} & \frac{1}{15} \\
      -\frac{2}{15} & \frac{14}{15} & -\frac{2}{15} \\
      \frac{1}{15} & -\frac{2}{15} & \frac{11}{15}
    \end{pmatrix},
    \\
  \renewcommand{\arraystretch}{1.5}
    \mat{R} = 
    \begin{pmatrix}
      \frac{2 + \sqrt{2}}{2} & -1 & \frac{2 - \sqrt{2}}{2} \\
      \frac{2 - \sqrt{2}}{2} & -1 & \frac{2 + \sqrt{2}}{2}
    \end{pmatrix},
    \quad
    \mat{D} = 
    \begin{pmatrix}
      -\frac{3 \sqrt{2}}{2} & 2 \sqrt{2} & -\frac{\sqrt{2}}{2} \\
      -\frac{\sqrt{2}}{2} & 0 & \frac{\sqrt{2}}{2} \\
      \frac{\sqrt{2}}{2} & -2 \sqrt{2} & \frac{3 \sqrt{2}}{2}
    \end{pmatrix}.
  \end{gather}
\end{itemize}

Additionally, the diagonal-norm nodal bases are
\begin{itemize}
  \item 
  Gauß-Legendre basis with matrices
  \begin{equation}
  \renewcommand{\arraystretch}{1.5}
    \mat{M} = 
    \begin{pmatrix}
      \frac{5}{9} & 0 & 0 \\
      0 & \frac{8}{9} & 0 \\
      0 & 0 & \frac{5}{9} \\
    \end{pmatrix},
    \quad
    \mat{R} = 
    \begin{pmatrix}
      \frac{5 + \sqrt{15}}{6} & -\frac{2}{3} & \frac{5 - \sqrt{15}}{6} \\
      \frac{5 - \sqrt{15}}{6} & -\frac{2}{3} & \frac{5 + \sqrt{15}}{6}
    \end{pmatrix},
    \quad
    \mat{D} = 
    \begin{pmatrix}
      -\frac{\sqrt{15}}{2} & \frac{2 \sqrt{15}}{3} & -\frac{\sqrt{15}}{6} \\
      -\frac{\sqrt{15}}{6} & 0 & \frac{\sqrt{15}}{6} \\
      \frac{\sqrt{15}}{6} & -\frac{2 \sqrt{15}}{3} & \frac{\sqrt{15}}{2}
    \end{pmatrix}.
  \end{equation}

  \item 
  Lobatto-Legendre basis with matrices
  \begin{gather}
  \renewcommand{\arraystretch}{1.5}
    \mat{M} = 
    \begin{pmatrix}
      \frac{1}{3} & 0 & 0 \\
      0 & \frac{4}{3} & 0 \\
      0 & 0 & \frac{1}{3}
    \end{pmatrix},
    \quad
    \mat{R} = 
    \begin{pmatrix}
      1 & 0 & 0 \\
      0 & 0 & 1
    \end{pmatrix},
    \quad
    \mat{D} = 
    \begin{pmatrix}
      -\frac{3}{2} & 2 & -\frac{1}{2} \\
      -\frac{1}{2} & 0 & \frac{1}{2} \\
      \frac{1}{2} & -2 & \frac{3}{2} \\
    \end{pmatrix}.
  \end{gather}
\end{itemize}

\section{Numerical solutions for Burgers' equation}
\label{sec:num-sol-Burgers}

Here, the numerical solutions corresponding to the energy and momentum in
Figure~\ref{fig:burgers-momentum-energy} of section \ref{sec:numerical-tests-Burgers}
are shown.
The values of $u(3)$ are in general similar -- two approximately affine-linear
parts and a discontinuous part with oscillations around $x = 1$. Despite of this,
the intensity of oscillations depends on the bases and associated projection used
for multiplication.

In this case, the Gauß-Legendre nodes and modal Legendre polynomials seem to be
visually indistinguishable. Contrary, the computations using a nodal basis are
much more efficient, since only simple multiplication of nodal values has to be
performed.

\begin{figure}[!hp]
  \centering
  \begin{subfigure}[b]{0.45\textwidth}
    \includegraphics[width=\textwidth]%
      {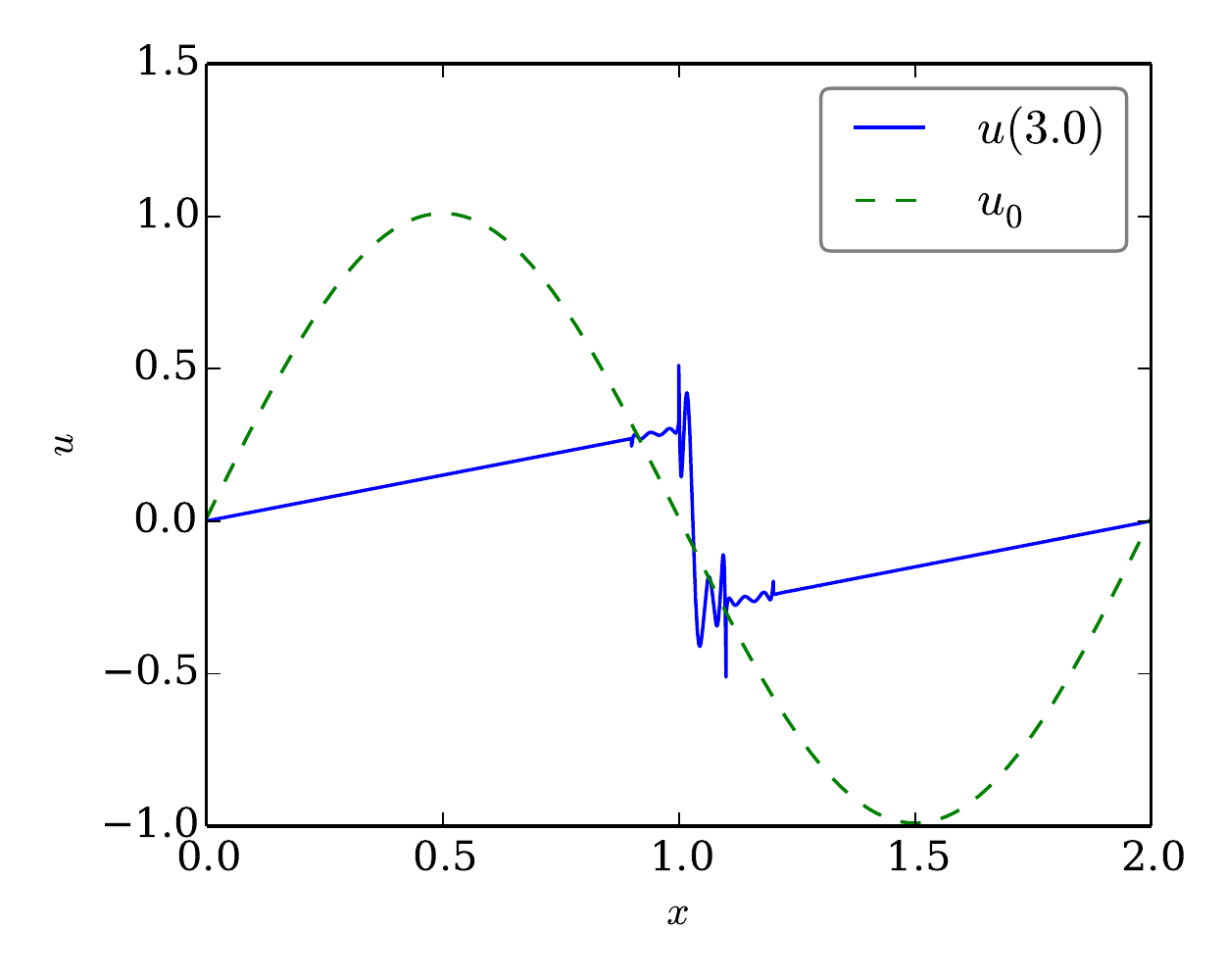}
    \caption{Gauß-Legendre.}
  \end{subfigure}%
  ~
  \begin{subfigure}[b]{0.45\textwidth}
    \includegraphics[width=\textwidth]%
      {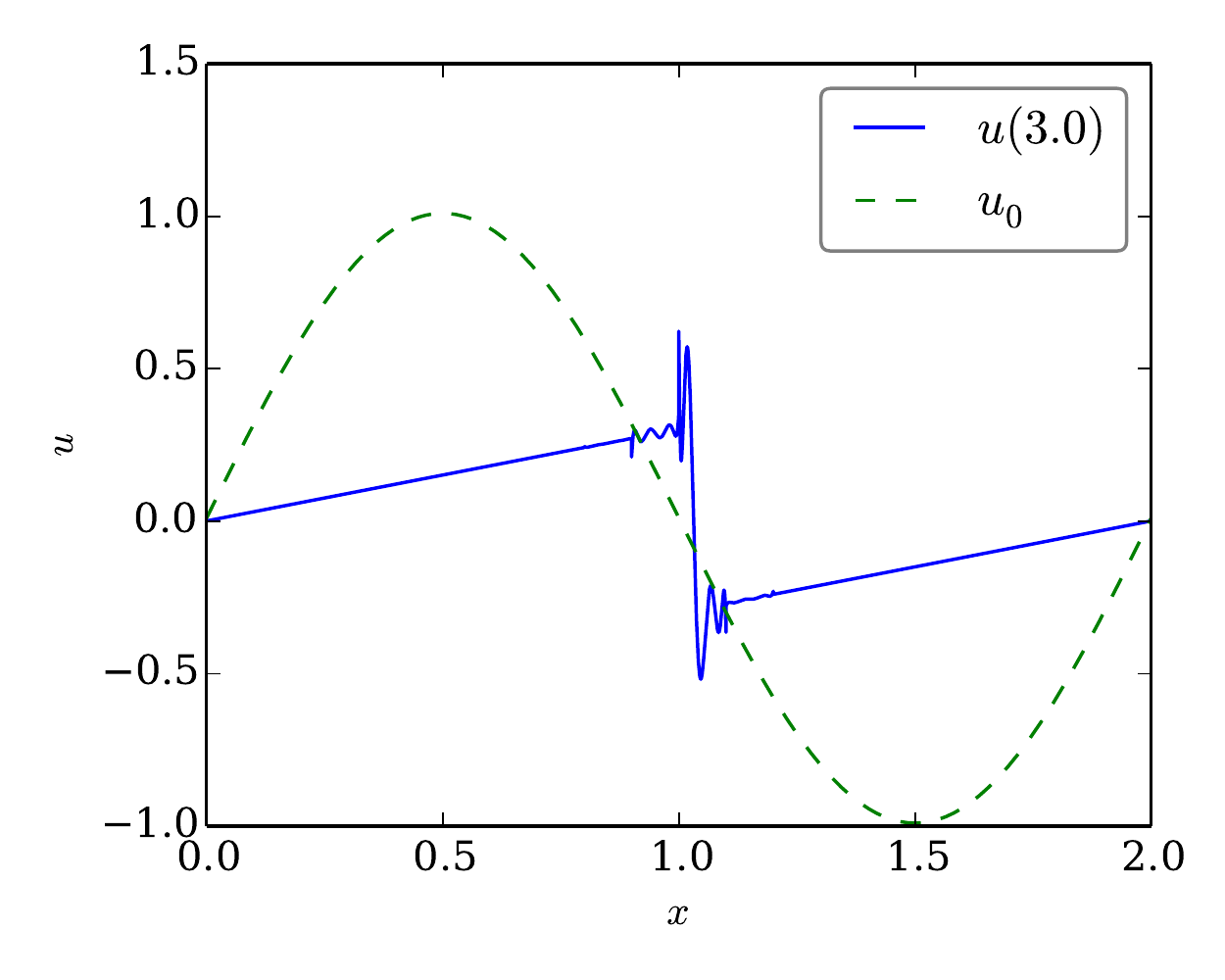}
    \caption{Lobatto-Legendre.}
  \end{subfigure}%
  \\
  \begin{subfigure}[b]{0.45\textwidth}
    \includegraphics[width=\textwidth]%
      {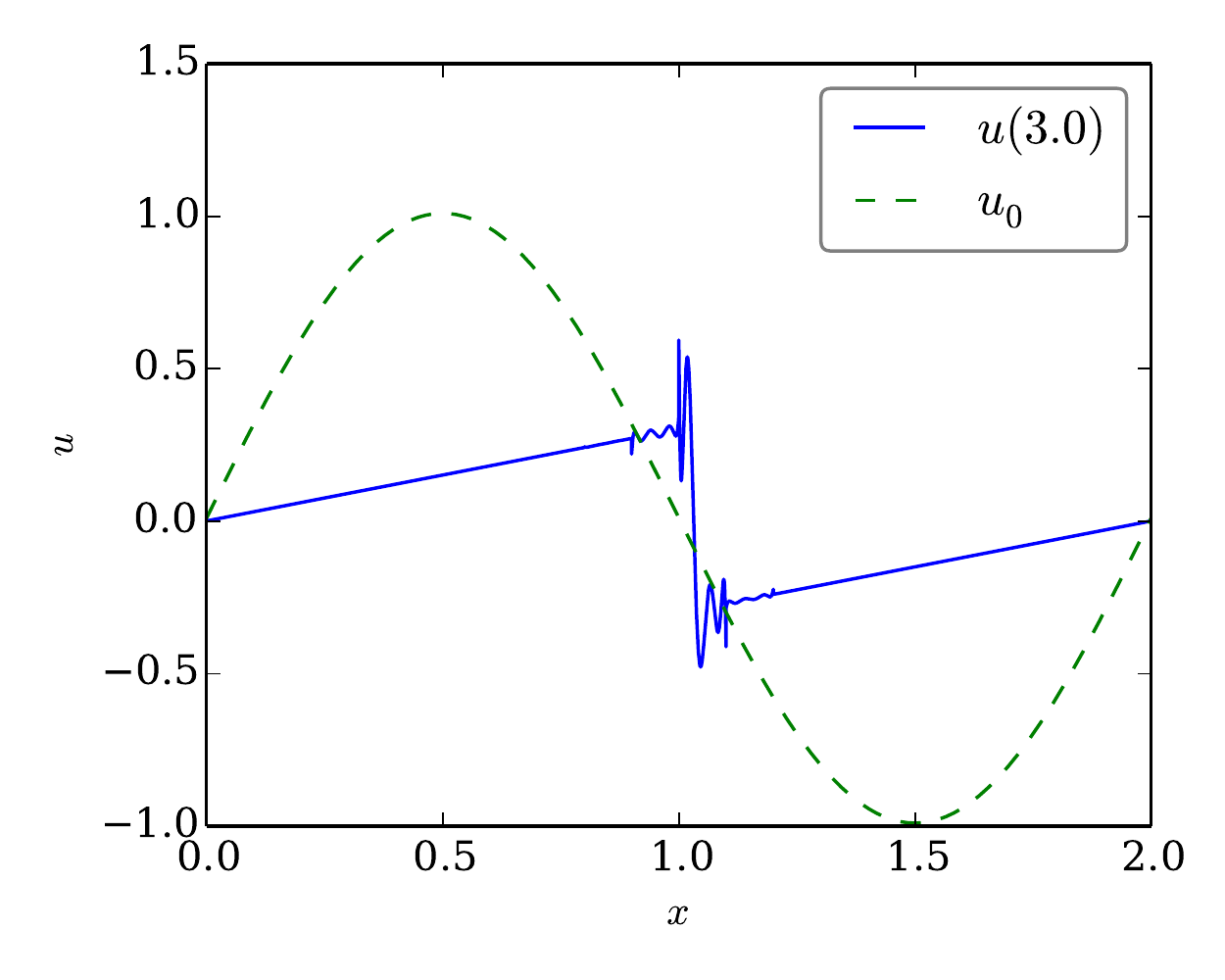}
    \caption{Chebyshev first kind, roots.}
  \end{subfigure}%
  ~
  \begin{subfigure}[b]{0.45\textwidth}
    \includegraphics[width=\textwidth]%
      {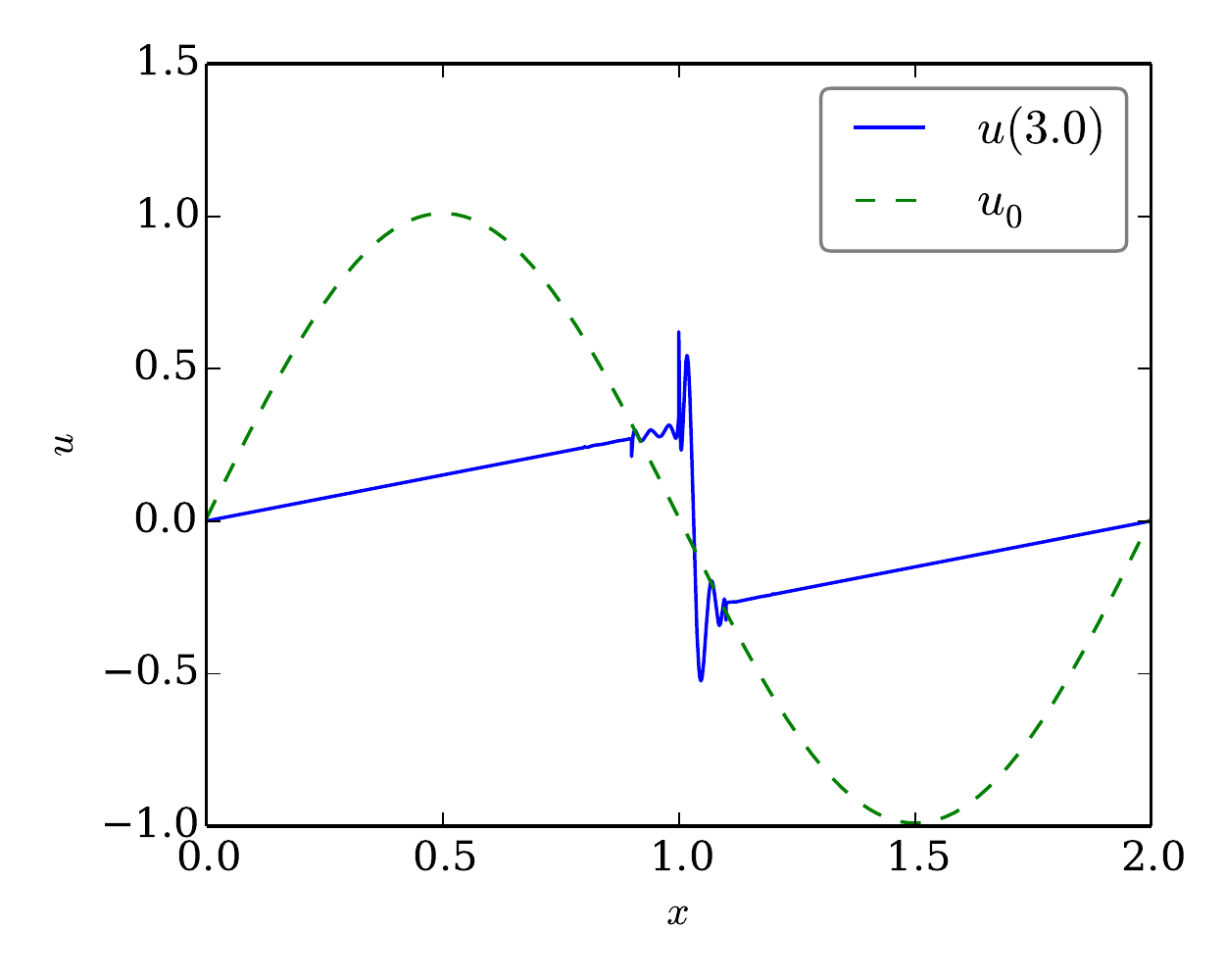}
    \caption{Chebyshev first kind, extrema.}
  \end{subfigure}%
  \\
  \begin{subfigure}[b]{0.45\textwidth}
    \includegraphics[width=\textwidth]%
      {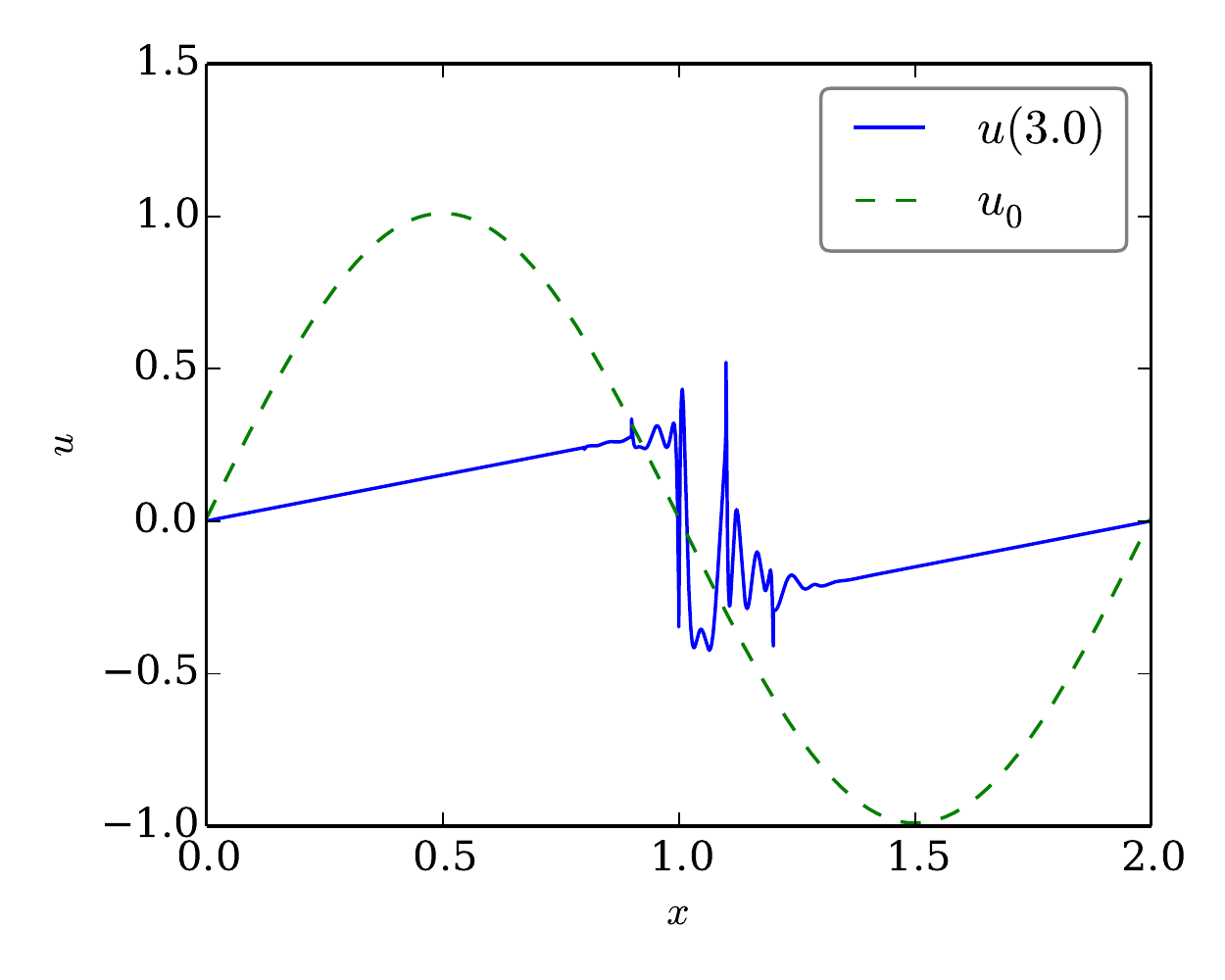}
    \caption{Chebyshev second kind, roots.}
  \end{subfigure}%
  ~
  \begin{subfigure}[b]{0.45\textwidth}
    \includegraphics[width=\textwidth]%
      {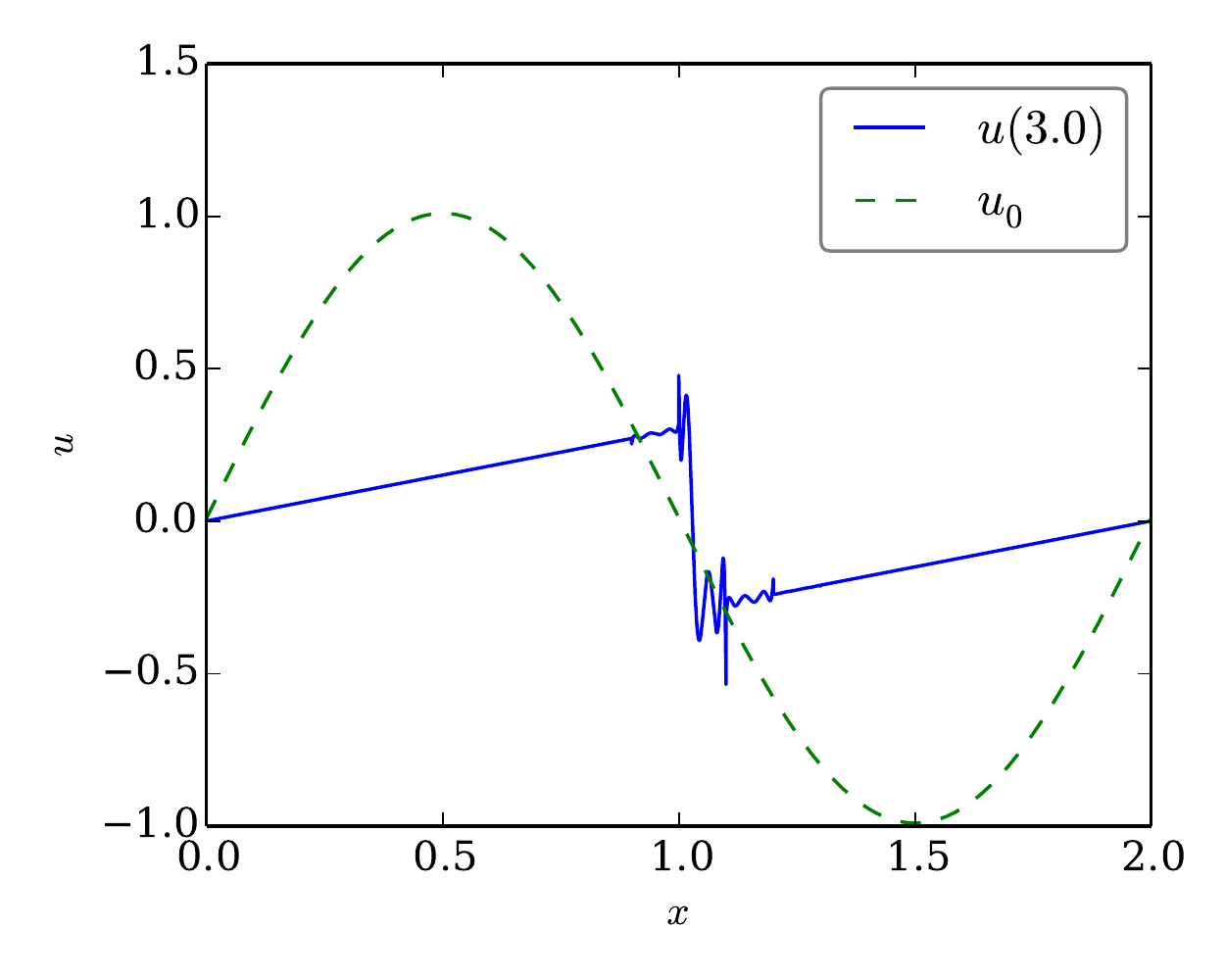}
    \caption{Legendre.}
  \end{subfigure}%
  \caption{Results of the simulations for Burgers' equation
           using general SBP CPR methods with 20 elements, different bases of
           order 7 and local Lax-Friedrichs (LLF) flux.
           Corrections for both divergence and restriction are used.
           Each Figure shows the values of $u(3)$ (blue line) and $u(0) = u_0$
           (green, dashed) for different bases.
           (For interpretation of the references to colour in this figure legend,
           the reader is referred to the web version of this article.)}
  \label{fig:burgers-solution}
\end{figure}

\section*{Acknowledgements}
The authors would like to thank the anonymous reviewers for their helpful comments,
resulting in an improved presentation of this material.

\bibliographystyle{model1b-num-names}
\bibliography{references}

\end{document}